\documentclass[11pt]{article}
\usepackage[margin=1.25in]{geometry}
\usepackage{genpreamble}

\title{Foliations forcing closed orbits}
\author{Ellis Buckminster and Audrey Rosevear}

\date{\today}

\begin{document}

\maketitle

\begin{abstract}
    We study foliations which force every transverse flow to have a closed orbit.  We fully characterize this property among $C^2$ finite depth foliations on atoroidal 3-manifolds. We also show that for many foliations, this property can be broken by a $C^0$-small modification of the foliation. We produce closed orbits using the theory of endperiodic maps, and we construct $C^1$ flows without closed orbits by inserting dynamical plugs transverse to foliations along curves of attracting holonomy.
\end{abstract}

\section{Introduction}
The topology and geometry of 2-dimensional foliations on 3-manifolds often enforce dynamical properties of transverse flows, and vice versa. For example, Sullivan showed a $C^2$ foliation is taut if and only if it has a transverse flow which is volume-preserving  (\cite{sullivan1979homological}, see also \cite{colin_taut_2019}). Zeghib showed that any flow transverse to a fibration on a closed 3-manifold is quasigeodesic \cite{zeghib1993feuilletages}. Zung  showed that a $C^2$ foliation admits a transverse Reeb flow if and only if it has no closed leaves and is not approximated by fibrations \cite{zung_reeb_2024}. We also mention the Gabai--Mosher conjecture, that taut foliations on hyperbolic 3-manifolds admit almost transverse pseudo-Anosov flows. 

Here, we tackle the following flavor of question: what properties of a foliation guarantee that \emph{all} transverse flows share some common dynamical property? In particular, we concern ourselves with one of the simplest dynamical properties, the existence of a closed orbit.

\begin{question}\label{q:mainq}
    What properties of a foliation force all transverse flows to have a closed orbit?
\end{question}

For a foliation by fibers on a hyperbolic 3-manifold, a theorem of Thurston \cite{thurston1988geometry} implies that any transverse flow has at least as many closed orbits as the pseudo-Anosov suspension flow, which has infinitely many. Therefore, such a foliation forces closed orbits in all transverse flows. It was known to Novikov \cite{Nov65} that if $\mc F$ has a Reeb component, then all transverse flows have closed orbits, see also Lemma \ref{lem:reebcomponentforces}. This question is also parallel to other existence questions for closed orbits of families of flows. Frankel proved quasigeodesic flows always have closed orbits \cite{frankel_closedorbits}. Taubes proved the Weinstein conjecture for Reeb flows in dimension three \cite{Taubes2007}, which states that all Reeb flows contain closed orbits. Indeed, the longstanding connections between contact geometry and foliation theory in dimension 3 (see \cite{colin_foliations_2020} for an introduction) were one source of inspiration for this paper.
\medskip

Gabai \cite{Gab83} showed that every hyperbolic 3-manifold with positive first Betti number admits a \emph{finite depth foliation}; hence, these provide an abundant supply of taut foliations. Finite depth foliations relax the notion of a foliation by fibers by allowing iterated spiraling of leaves (see Section \ref{sec:background} for precise definitions). We address Question \ref{q:mainq} for foliations in closed, atoroidal 3-manifolds, primarily in the case of finite depth foliations. Much of our motivation comes from the case of foliations by fibers in hyperbolic 3-manifolds, so our setting is a natural next step.

\subsection*{Results}
 Throughout this section, let $M$ be a closed, connected, oriented, atoroidal 3-manifold, and let $\mc F$ be a 2-dimensional foliation on $M$. We say $\mc F$ \defn{forces closed orbits} if all transverse flows have closed orbits. An \defn{aperiodic flow} is a flow without closed orbits. Our first result is a complete answer to Question \ref{q:mainq} in the case of finite depth foliations which are $C^2$-smoothable. 

\begin{mainthm}[Characterization Theorem]\label{thma:classification}
    Suppose $\mc F$ is a finite depth, $C^2$ foliation. Then $\mc F$ forces closed orbits if and only if 
    \begin{enumerate}
        \item $\mc F$ contains a locally top-depth region that is not a topological product, or
        \item $\mc F$ is $C^0$-approximated by such foliations.
    \end{enumerate}
    If $\mc F$ does not force closed orbits, it admits a $C^1$ transverse aperiodic flow.
\end{mainthm}

Here, a \emph{locally top-depth region} is a connected set which is a union of top-depth leaves in its closure; see Definition \ref{def:ltdregion}. We reinterpret the second condition in terms of what we call having \emph{connecting junctures}. Intuitively, $\mc F$ has connecting junctures if the spiraling of depth $k$ leaves on either side of each depth $k-1$ leaf ``lines up''. See Section \ref{sec:junctures} for precise definitions.

\medskip

Given a foliation, one can \emph{shear} it to produce a new foliation by making a modification in a small neighborhood of an annulus or infinite strip; see Section \ref{sec:background finite depth} for details. The following result says that given any finite depth foliation $\mc F$ on an atoroidal 3-manifold, one can always obtain a foliation that forces closed orbits by iteratively ``unshearing'' $\mc F$.

\begin{maincor}\label{cor:blowdown}
    Suppose $\mc F$ is a finite depth foliation. If $\mc F$ does not force closed orbits, then it is obtained by iterated shearing of a foliation that does.
\end{maincor}

We also obtain the following corollary, which highlights that the primary mechanism for forcing closed orbits is trivial holonomy. The same result applies to foliations obtained by taking ``irrational perturbations'' of non-product locally top-depth regions to obtain regions without holonomy; see \cite[Chapter 9]{candel_foliations_1999} for more details on these foliations.

\begin{maincor}\label{cor:no holonomy}
    If $\mc F$ is $C^2$ and has no holonomy, then $\mc F$ forces closed orbits.
\end{maincor}

Our second main theorem shows that in the general setting, forcing closed orbits is often unstable under small perturbations of the tangent plane field. In particular, forcing closed orbits is only a property of the ambient 3-manifold when the manifold admits no taut foliation. Theorem \ref{thma:classification} implies that if $\mc F$ is depth $k$, then the approximating foliations can be chosen to be depth $k+1$.

\begin{mainthm}[Instability Theorem]\label{thma:perturb}
    If every leaf of $\mc F$ has genus, then $\mc F$ can be $C^0$-approximated by foliations which do not force closed orbits.
\end{mainthm}

\medskip

In previous work (see e.g. \cite[Theorem 14.20]{farb2011primer}, \cite{landry2023endperiodic}, \cite{buckminster2025periodic}), one guarantees periodic points of surface homeomorphisms by using the Lefschetz fixed point theorem \cite{lefschetz1926intersections} applied to the compactified universal cover of the surface. As a corollary of our work, the only guaranteed closed orbits of flows transverse to finite depth foliations arise as periodic points of induced monodromies (or induced `approximate' monodromies). These periodic points are all guaranteed by the Lefschetz fixed point theorem applied to the universal cover of a leaf. Thus, if a finite depth foliation forces closed orbits in all transverse flows, a closed orbit can always be produced via surface dynamics.

\subsection*{Historical context and motivation}   
\subsubsection*{Fibered hyperbolic 3-manifolds}
A key motivating example for this work is the picture of fibered hyperbolic 3-manifolds developed by Thurston and Fried. Given a closed, hyperbolic 3-manifold $M$ fibering over the circle, the different ways $M$ fibers are organized by fibered cones of the Thurston norm ball in $H_2(M,\R)$ \cite{thurston1986norm}. Each fibered cone corresponds to a pseudo-Anosov suspension flow $\phi$, in the sense that the suspension flow of any fibration in that cone is orbit equivalent to $\phi$ \cite{fried_geometry_1982}. 

Each fibration yields a foliation $\mc F$, which is the foliation by fibers. The flow $\phi$ for the cone containing  $\mc F$ is transverse to $\mc F$, and this transverse flow yields geometric information about $\mc F$, such as the existence of a Cannon--Thurston map \cite{CannonThurston, fenley2012ideal}. Furthermore, any flow $\phi$ transverse to $\mc F$ is a suspension flow, and thereby induces a homeomorphism from any leaf of $\mc F$ to itself. Therefore, the Nielsen--Thurston classification can be applied to obtain dynamical information about $\phi$, such as the existence of an abundance of closed orbits (see e.g. \cite{fathi_thurstons_2013}).

\subsubsection*{Big mapping class groups} This work also connects to the very active field of infinite-type surfaces and big mapping class groups. Other than the compact leaves, the leaves in our finite depth foliations are infinite-type surfaces. In top-depth components, a transverse flow induces a first-return map on these leaves. These first-return maps are similar to endperiodic maps, though leaves of depth 2 or higher will typically have infinitely many ends.  In the depth one case, we get true endperiodic first-return maps. In the case where all depth one regions are topological products and the junctures do not connect, the first return maps of the aperiodic flows we produce are homotopic to translations and have compact invariant sets with no periodic points; see Section \ref{subsec:firstreturn} for a description of these maps.

\subsubsection*{Aperiodic flows and plugs} A third source of motivation for this project is the study of the wide-ranging dynamical behaviors exhibited by flows on 3-manifolds, and in particular the existence and constructions of aperiodic flows. The \emph{Seifert Conjecture} \cite{seifert1950closed} stated that all nonsingular vector fields on $S^3$ contain closed orbits; it is now known that every closed 3-manifold admits an aperiodic flow. Wilson \cite{Wilson1966OnTM} produced a $C^\infty$ flow with finitely many closed orbits, and aperiodic flows of regularities $C^1$, $C^2$, and $C^\infty$  were produced by Schweitzer \cite{Schweitzer}, Harrison \cite{HARRISON1988249}, and K. Kuperberg \cite{kkuperberg}, respectively, along with a $C^1$ volume-preserving construction due to G. Kuperberg \cite{kuperberg_volume-preserving_1995}. All of these results proceed via the construction of \emph{plugs}. These are vector fields on 3-manifolds with boundary that, when inserted into an initial 3-manifold with a vector field, can break closed orbits of the original flow without creating any new closed orbits. See Section \ref{sec:other plugs} for a discussion on using higher-regularity and volume-preserving plugs in this setting.

\subsection*{Proof techniques}
The proof of the positive direction of Theorem \ref{thma:classification} in the depth one case goes as follows. When $\mc F$ contains a depth one region that is not a topological product, the endperiodic maps machinery of Handel-Miller (written down by Cantwell--Conlon--Fenley \cite{cantwell_endperiodic_2021}) and Landry--Minsky--Taylor \cite{landry2023endperiodic} produces a closed orbit of any transverse flow. When $\mc F$ is made entirely of topological products with connecting junctures, it is approximable by fibrations over the circle, and any transverse flow is therefore a suspension flow. The higher depth case is similar. The only place we use the $C^2$ hypothesis is to characterize foliations approximable by foliations with non-product top-depth regions. Our characterization of foliations approximated by fibrations is in the spirit of the constructions of Thurston \cite{thurston1986norm} and Landry--Minsky--Taylor \cite{landry2023endperiodic}, and partially extends characterizations by Andrade--Firmo \cite{andrade_foliations_2003} of foliations which are limits of \textit{isotopic} fibrations.

Our work is primarily concerned with the negative direction, where the junctures do not connect. In this case, we show that there are many leaves with curves of two-sided attracting holonomy, and produce an aperiodic transverse flow by using a modified version of the \emph{Schweitzer plug} \cite{Schweitzer}. A key technical result of the paper is Lemma \ref{lem:pluginsertion}, which shows that plugs can be inserted transverse to the foliation along nonseparating curves of two-sided attracting holonomy. One novelty of our work is a strategy to produce aperiodic flows with a small number of plugs (in some cases, just one), starting with a flow whose chain recurrent set is 1-dimensional.

Theorem \ref{thma:perturb} is proved by inserting lots of very thin topological product regions containing Schweitzer plugs, similar to the negative direction of Theorem \ref{thma:classification}.

\subsection*{Organization} 
In Section \ref{sec:background}, we review background from the theory of foliations. In Section \ref{sec:forcingclosedorbits}, we discuss techniques to force closed orbits using endperiodic maps and prove Corollary \ref{cor:no holonomy}. In Section \ref{sec:plugs}, we demonstrate how to break closed orbits of a flow by inserting plugs and give conditions on when plugs can be inserted transverse to foliations. This allows us to prove Theorem \ref{thma:perturb}. In Section \ref{sec:junctures}, we introduce the connecting (limit) junctures property and characterize $C^2$ finite depth foliations with this property. This allows us to show that for a given foliation, either the positive techniques from Section \ref{sec:forcingclosedorbits} or the negative techniques from Section \ref{sec:plugs} apply. In Section \ref{sec:constructingflows}, we prove Theorem \ref{thma:classification} and Corollary \ref{cor:blowdown}.

\subsection*{Acknowledgments} We would like to thank our advisors Sam Taylor, Ian Agol, and Michael Hutchings, as well as Priyam Patel for helpful discussions over the course of this project. We would especially like to thank Rafael Potrie for pointing us in the direction of Smale flows and helping us with the relevant literature. We would also like to thank Siddhi Krishna, Rafael Potrie, and Sam Taylor for comments on an earlier draft.

The completion of this paper was supported by the National Science Foundation under Grant No. DMS–2424139, while the authors were in residence at the Simons Laufer Mathematical Sciences Institute in Berkeley, California, for the Spring 2026 semester program \emph{Topological and Geometric Structures in Low Dimensions}. The first author was partially supported by a Simons Dissertation Fellowship [SFI-MPS-SDF-00014998, EB]. The second author was partially funded by a grant from the Simons Foundation [376200, IA], and by the Berkeley Connect Fellowship.

\medskip\noindent\textbf{Declaration of AI use:} Generative AI was used for spell checking and nothing else.

\section{Background}\label{sec:background}

\subsection{Foliations and Flows}
By a (2-dimensional) \defn{foliation} $\mc F$ on a 3-manifold $M$ we mean a partition of $M$ into 2-dimensional injectively immersed submanifolds, called \defn{leaves}, and a collection of \defn{foliation charts} covering $M$, that is, open sets $U$ and homeomorphisms $U \to \R^2 \times \R$ taking leaves to parallel copies of $\R^2$. 

To a foliation $\mc F$ we associate the tangent plane field $T\mc F$. We say that a sequence $\mc F_i$ of foliations \defn{`$C^0$-approximates $\mc F$'} if the plane fields $T \mc F_i$ approach $T\mc F$ in the $C^0$ topology. We always assume our foliations are orientable and coorientable.

Calegari showed that any foliation on a 3-manifold can be isotoped so that the leaves are $C^\infty$-injectively immersed \cite{calegari_leafwise_2001}. For convenience, we will assume all foliations are in this form. We assume that the foliation chart homeomorphisms are smooth in the first coordinate. By a \defn{`$C^k$ foliation'} we mean a foliation such that the first coordinate of each foliation chart is $C^\infty$ and the second coordinate is $C^k$.  By a \defn{`$C^k$-smoothable foliation'} we mean a foliation $\mc F$ such that there exists a homeomorphism of $M$ taking $\mc F$ to a $C^k$ foliation.

\medskip

By a \defn{flow} $\phi = \phi^t$ on $M$ we mean a 1-parameter group of homeomorphisms $\phi^t(x): \R \times M \to M$. All flows in this paper will be of sufficient regularity to be defined by a unique vector field $V$. When we say a $C^k$ flow we mean that it is generated by a $C^k$ vector field $V$. We always assume our flows have no stationary points, equivalently, we assume $V$ is nonvanishing. 

When we say \defn{`$\phi$ is transverse to a foliation $\mc F$'} we mean that $V$ is transverse to $T \mc F$. We always assume that $\phi$ is positively transverse to $\mc F$ with respect to the coorientation of $\mc F$. When $\mc F$ is coorientable, such a flow always exists.

\subsection{Saturated sets and sutured manifolds}

A \defn{saturated set} of a foliation $\mc F$ is a set which is a union of leaves of the foliation. To an open, saturated set $U$ we associate the \defn{metric completion} $\wh U$, which is the completion of $U$ under the induced path metric coming from any metric on the ambient 3-manifold. This is a (typically noncompact) manifold with boundary. The completion $\wh U$ inherits a foliation on the interior from $\mc F$, which extends uniquely to a foliation on $\wh U$ by taking the boundary surfaces as leaves. This foliation has the same regularity as $\mc F$. A coorientation on the foliation induces a decomposition $\partial \wh U = \partial_- \wh U \sqcup \partial_+ \wh U$, where a positively transverse vector field points inward along $\partial_- \wh U$ and outward along $\partial_+ \wh U$.

There is a canonical immersion $\wh U \to M$ which takes the boundary leaves of $\wh U$ to leaves of $\mc F$. This immersion is an embedding on the interior, and is either an embedding on all of $\wh U$ or identifies some number of boundary leaves of $\wh U$. We denote by $\mc F_{\wh U}$ the pullback of $\mc F$ along this immersion, which is precisely the induced foliation. If $\phi$ is a flow transverse to $\mc F$, we denote by $\phi_{\wh U}$ the partially defined pullback flow under this immersion. See \cite{candel_foliations_1999}, Section 5.2 for details.

We say that an open, saturated set $U$ is a \defn{topological product region}, or just a \defn{product region}, if it is homeomorphic to $L \times (0, 1)$, where $L$ is a boundary leaf of $\wh U$. For a product region $U$, we say that it is a \defn{trivially foliated product} if $\mc F$ restricts to the foliation by parallel copies of $L$. If $\phi$ is a flow transverse to $\mc F$, we say a product region $U$ is a \defn{dynamical product} if the flowlines of $\phi$ restrict to the interval foliation by parallel copies of $(0, 1)$. Equivalently, $\phi$ has no invariant sets contained in $U$. Every foliated product is a dynamical product for any transverse flow, but the converse is not true. Note that this differs from the terminology Fenley--Mosher use in \cite{fenley1995quasigeodesic}.

\medskip

A \defn{sutured manifold} $N$ is a compact manifold with corners and a decomposition of its boundary and corners into compact surfaces $\partial_+N,$ $\partial_- N,$ and $ \partial_\pitchfork N$, such that:
\begin{enumerate}
    \item $\partial_+ N$ and $\partial_- N$ are disjoint and meet the corners of $N$ only at their boundaries
    \item $\partial_\pitchfork N$ is a union of annuli with boundary on the corners of $N$. Each annulus has one boundary component on $\partial_+ N$ and one on $\partial_- N$.
\end{enumerate}

Sutured manifolds were defined by Gabai \cite{Gab83}, though our definition is slightly different. In particular, we use manifolds with corners, and we do not allow toral sutures.

Let $i: N \to M$ be an immersion of a sutured manifold which is injective on the interior. We say that $N$ (with the immersion understood) \defn{carries} a foliation $\mc F$ if $\partial_- N$ and $\partial_+ N$ are contained in leaves of $\mc F$, with a positively transverse vector field pointing into $N$ along $\partial_- N$ and out of $N$ along $\partial_+ N$, and if $\partial_\pitchfork N$ is transverse to $\mc F$. We say that $N$ is \defn{compatible} with a flow $\phi$ transverse to $\mc F$ if the flowlines of $\phi$ are tangent to $\partial_\pitchfork N$ and trace out a foliation of each annular component by interval fibers. Our orientation convention implies that these interval flowlines start at the intersection of $\partial_\pitchfork N$ and $\partial_- N$ and end at the intersection of $\partial_\pitchfork N$ and $\partial_+ N$. Given a carrying sutured manifold, we denote by $\mc F_N$ the pullback of $\mc F$ under this immersion. Given a compatible transverse flow $\phi$, we denote by $\phi_N$ the pullback partially defined flow on $N$.

We will use the following elementary but foundational lemma, the Octopus Decomposition lemma. See \cite{candel_foliations_1999}, Section 5.2 for a proof. We have rephrased the lemma in the language of sutured manifolds.

\begin{lemma}[Octopus Decomposition]\label{lem:octopus}
    Let $U$ be an open saturated set of $\mc F$ on a compact manifold $M$. Then the completion $\wh U$ admits an \defn{octopus decomposition} $\wh U = N \cup A_1 \cup ... \cup A_n$, where $N$ is a sutured manifold carrying $\mc F_{\wh U}$ such that $\partial_\pm N \subset \partial_\pm \wh U$, and each $A_i$ is a foliated interval bundle over its noncompact boundary leaves $L_i^\pm$, which are contained in $\partial \wh U$ and intersect $N$ along $\partial (\partial_\pm N)$. Moreover, $A_i$ and $N$ intersect along a component of $\partial_\pitchfork N$.

    If $\phi$ is a flow transverse to $\mc F$, the octopus decomposition can be chosen so that $N$ is compatible with $\phi_{\wh U}$ and $\phi$ defines a dynamical product on each $A_i$. We call such an octopus decomposition `\defn{compatible with $\phi$}'.
\end{lemma}

\subsection{Finite depth foliations}\label{sec:background finite depth}

A leaf $L$ of a foliation $\mc F$ is \defn{depth 0} if it is compact, and \defn{depth $k$} if $\overline{L} \setminus L$ is a union of leaves of depth at most $k -1$, with at least one leaf of depth exactly $k -1$. The foliation $\mc F$ is \defn{finite depth} if there is some $K$ such that all leaves of $\mc F$ are of depth less than or equal to $K$. We say that $\mc F$ is \defn{depth $k$} if the maximum of the depth of all leaves of $\mc F$ is exactly $k$.

An open, saturated set $U$ is of \defn{relative depth $k - \ell$} if the highest depth of a leaf contained in $U$ is $k$ and the lowest is $\ell$. We make a similar definition for its completion $\wh U$. Note that the relative depth of $\wh U$ may be different than that of $U$. 

A simple way to increase the depth of a finite depth foliation $\mc F$ is to \defn{shear} it, defined as follows. Let $U\cong L\times [0,1]$ be a topological product that is foliated as a product, i.e. with leaves of the form $L\times \{t\}$ for $t\in[0,1]$. Let $\gamma$ be either a simple closed curve or a properly embedded line in $L$, and let $A$ be the annulus or infinite strip obtained by extruding $\gamma$ in $U$, i.e. 
$$A=\gamma\times[0,1]\subset U.$$
We shear $\mc F$ along $A$ by cutting along $A$ and regluing by a monotone homeomorphism $g\colon [0,1]\to [0,1]$, i.e. identifying $(p,t)\in A$ with $(p,g(t))$. This will create a higher depth region spiraling onto $L\times \{1\}$ and $L \times \{0\}$ with juncture class Poincar\'e dual to $\pm\gamma$ (see Section \ref{sec:background holonomy junctures} for the definition of a juncture class).

To simplify product regions that are not necessarily trivially foliated products (such as those arising from shearing), we sometimes use an auxiliary transverse flow to \defn{blow down} these product regions as follows. Let $V$ be an open product region, and let $\phi$ be a smooth flow transverse to $\mc F$ and defining a dynamical product $V$. Following the proof of Proposition 2.1 in \cite{fenley1995quasigeodesic}, we first `thicken' any leaves in $\partial \wh V$ that are isolated among leaves of the same depth by replacing them with thin trivially foliated products. We then collapse $\wh V$ along flow lines of $\phi$. 

\subsection{Holonomy and junctures}\label{sec:background holonomy junctures}
Let $L$ be a leaf of a foliation $\mc F$ and let $\gamma$ be a closed curve in $L$ based at a point $x$. Let $\tau: [0, \varepsilon) \to M$ be a small, embedded interval positively transverse to $\mc F$ with $\tau(0) = x$. Then ``sliding $\tau$ along $\gamma$'' defines the \defn{holonomy map} 
\[
h_\gamma: I \to \tau,
\]
where $I \subset [0, \varepsilon)$ is a relatively open subinterval. This map is continuous and injective, and its regularity is the same as the regularity of $\mc F$. See \cite[Section 8.2]{candel_foliations_1999} for precise definitions. We always have that $x \in I$ and $h(x) = x$. This map only depends on $\gamma$ up to homotopy. We say that $\gamma$ has \defn{attracting holonomy on the positive side} if $h(y) < y$ for all $y$ sufficiently close to $x$. This property is independent of the choice of $\tau$ or of the basepoint $x$. For $\tau: (-\varepsilon, 0] \to M$ positively transverse to $\mc F$ with $\tau(0) = x$, we similarly define $\gamma$ to have \defn{attracting holonomy on the negative side}. We say that $\gamma$ has \defn{two-sided attracting holonomy} or just \defn{attracting holonomy} if it has attracting holonomy on both the positive and negative sides.

\medskip

Suppose $\mc F$ has two leaves $L_-$ and $ L_+$ which cobound a product region $V$ on the positive side of $L_-$ and the negative side of $L_+$. Let $A$ be an embedded annulus in $V$ which is transverse to $\mc F$ with boundary components $\partial_- A$ and $\partial_+ A$ lying in $L_-$ and $L_+$, respectively. We say that $A$ is an \defn{annulus of attracting holonomy} if $\partial_- A$ has attracting holonomy on the negative side and $\partial_+ A$ has attracting holonomy on the positive side. 

\medskip

Let $L$ be a depth $k$ leaf of a foliation $\mc F$ and let $S$ be a depth $k - 1$ leaf in $\overline{L} \setminus L$ such that any transversal through $S$ intersects $L$ on the positive side. We say that $L$ \defn{spirals onto $S$ on the positive side}. We make a similar definition for spiraling on the negative side. A spiral on the positive side of $S$ defines a \defn{juncture class} in $H^1(S)$ as follows: fix a basepoint $x \in S$ and $\tau: [0, \varepsilon) \to M$ a positive transversal. Then if $\gamma$ is any loop based at $x$ and $y \in \tau([0, \varepsilon)] \cap L$, $h_\gamma (y)$ is still in $L$. Since $L$ is a proper leaf, the intersection of $\tau([0, \varepsilon))$ and $L$ defines a sequence $y_i$ converging monotonically towards $x$. Therefore there is some $p$ such that $h_\gamma (y_i) = y_{i + p}$ for all $i$ sufficiently large. By reversing $\gamma$ if necessary, we assume $p \geq 0$, and we fix $p$ as the smallest positive such $p$ taken over all homotopy classes of loops in $S$. Then there is a well-defined homomorphism $j: \pi_1(S) \to \Z$ such that $h_\gamma (y_i) = y_{i + j(\gamma)p}$ for all $\gamma$. This homomorphism is the juncture class, and juncture classes for spiralling on the negative side are defined similarly. We say that $j$ is \defn{compactly supported} if it lives in the compactly supported cohomology $H^1_c(S)$. We note that the juncture class, by definition, is always primitive.

\begin{figure}[ht]
	\centering
	\includegraphics[width=0.4\linewidth]{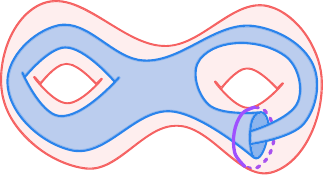}
	\caption{The purple curve is Poincar\'e dual to a juncture class on the orange surface.}
	\label{fig:juncture}
\end{figure}

The picture to keep in mind is Figure \ref{fig:juncture}, where the blue leaf $L$ spirals onto the orange leaf $S$ on the positive side, assuming the coorientation on $S$ points `outwards'. This determines a juncture class $j$ on $S$. The purple curve $\gamma$ is Poincar\'e dual to this juncture class, which we understand as follows. If $\alpha$ is a curve on $S$ with algebraic intersection number $p$ with $\gamma$, we observe that sliding a negative transversal $\tau$ along $\alpha$, we move `up' a level on $L$ once for every positive intersection of $\alpha$ with $\gamma$, and `down' a level for every negative intersection. Phrasing this in terms of intersections $\{y_i\}_{i\in\N}$ of $\tau$ with $L$, we get
\[
h_\alpha (y_i)=y_{i+p},
\]
so we have $j(\alpha)=p$.

Our main use of the $C^2$-smoothability hypothesis comes from the following theorem of Cantwell--Conlon \cite{cantwell_smoothability_1988, cantwell_topological_1994}:

\begin{theorem}[Cantwell--Conlon \cite{cantwell_smoothability_1988, cantwell_topological_1994}]
    Let $\mc F$ be finite depth. Then the following are equivalent:
    \begin{itemize}
        \item $\mc F$ is $C^2$-smoothable
        \item $\mc F$ is $C^\infty$-smoothable
        \item All junctures of all spirals of $\mc F$ are compactly supported.
    \end{itemize}
\end{theorem}

We note that if a juncture class of a positive spiral  evaluates positively on a loop $\gamma$, then $\gamma$ has attracting holonomy on the positive side.

If $S$ is a depth $k$ leaf, then all depth $k+1$ leaves spiraling onto $S$ on the positive side must have the same juncture class, up to a positive multiple. Call this the \defn{positive juncture class} $j^+(S)$, and similarly define the \defn{negative juncture class} $j^-(S)$. These classes do not typically agree.

\subsection{Locally top-depth components}
\begin{definition}\label{def:ltdregion}
    A \defn{locally top-depth component} $U$ of a foliation $\mathcal{F}$ is a connected, open, saturated subset, such that all leaves in $U$ have the same depth, which is finite and greater than the depth of any leaf in $\overline{U}$.
\end{definition}

Every finite depth foliation has at least one locally top-depth component, coming from a maximal connected set of maximum depth leaves. However, not all locally top-depth components are comprised of maximum depth leaves. A finite depth foliation can also have an infinite number of locally top-depth components; however, all but a finite number of them will be product regions if the ambient 3-manifold is compact.

The key feature we will use about locally top-depth components is that the foliation restricts to a foliation by fibers over the circle.

\begin{lemma}\label{lem:top depth is fibered}
    If $U$ is a locally top-depth component of a foliation $\mathcal{F}$, where each leaf in $U$ is homeomorphic to a surface $L$, then $\mathcal{F}\left.\right|_U$ fibers over $S^1$ with fiber $L$. 
    
    Conversely, if $U$ is a connected open saturated set of a finite depth foliation on which the foliation restricts to a fibration over $S^1$, then $U$ is a locally top-depth component.
\end{lemma}

\begin{proof}
    Suppose $U$ is a top-depth component. A combination of Lemma 9.1.3 and Theorem 9.1.4 in \cite{candel_foliations_1999}  implies that the leaf space $\mathcal{L}$ of $\mathcal{F}\left.\right|_U$ is a closed 1-manifold. The key observation is that all the leaves in $U$ are proper and have trivial germinal holonomy, so a neighborhood of each leaf of $U$ is a trivial product, giving the local Euclidean structure.
    
    Note that although the results in \cite{candel_foliations_1999} are stated for $C^2$ foliations, the primary place the assumption is used is to guarantee that the leaves of $U$ are proper, local minimal sets. However, by our hypothesis, all leaves of $U$ are proper and so the rest of the proofs of both results hold verbatim.

    The leaf space $\mathcal{L}$ will be a circle if and only if $U$ is not a trivially foliated product (\cite[Lemma 9.1.6]{candel_foliations_1999}), which follows from the definition of a locally top-depth component.

    On the other hand, suppose $U$ is a connected open saturated set in a finite depth foliation which is not a locally top-depth component. We claim that $\mc F$ does not restrict to a fibration over $S^1$ on $U$. If not all leaves of $U$ are homeomorphic, or if not all leaves in $U$ have the same depth, this is clear. So suppose $U$ satisfies all the required conditions except that there is some leaf $S$ in $\overline{U} \setminus U$ which has depth equal to the depth of leaves in $U$ (by definition, $S$ cannot have greater depth than $U$). In particular, this shows that the leaves of $U$ do not spiral onto $S$. Then $S$ is homeomorphic to the leaves of $U$ by an isotopy along flowlines of an auxiliary transverse flow. This is a consequence of the Dippolito Semistability Theorem, see \cite[Theorem 5.3.4]{candel_foliations_1999}. 
    
    The same argument as above shows that the leaf space of $\mc {F} \left.\right|_U$ is a 1-manifold. But $S$ provides an endpoint compactifying the leaf space of $U$ on one side. This shows that the leaf space is homeomorphic to an open interval, not $S^1$.
\end{proof}

We note there exist foliations which restrict to fibrations on open, saturated sets, for which the leaves of the fibrations are not finite depth.

The Octopus Decomposition Lemma (Lemma \ref{lem:octopus}) immediately yields the following for locally top-depth components.

\begin{lemma}\label{lem:depth1octopus}    Let $\phi$ be a flow transverse to a foliation $\mc F$ and let $U$ be a locally top-depth component. Then there is an immersed, connected, sutured manifold $N$, carrying $\widehat{U}$ and compatible with $\phi$, which is embedded away from $\partial_\pm N$. Moreover, $N$ satisfies the following properties:
    \begin{enumerate}
        \item $\partial_\pm N \subset \partial \widehat{U}$
        \item The pullback of $\mathcal{F}$ to $N$ is depth 1 with all depth 0 leaves contained in the boundary.
        \item The foliation by flowlines of $\phi$ gives the structure of an interval bundle in the complement of $N$ in $U$. In particular, every orbit of $\phi$ which enters $U$ outside of $N$ leaves $U$.
    \end{enumerate}
\end{lemma}
In the case that $U$ has depth 1, $N = \widehat{U}$.

If $U$ is a locally top-depth component, then each (positive or negative) boundary component of $\widehat{U}$ has an associated (negative or positive) juncture class from any leaf in the interior of $U$. This set of junctures determines the foliation in the interior, see \cite{cantwell_foliation_1999}. We will only need this fact in the case of a product region, where we will use the following slightly stronger result.

\begin{lemma}\label{lem:juncturesinproduct}
    Let $U$ be a locally top-depth component which is a product region, and let $\Sigma_\pm$ be the positive and negative boundary components of $\widehat{U}$. Then $j^+(\Sigma_-) = -j^-(\Sigma_+)$, and up to topological conjugacy $\widehat{U}$ is obtained by shearing the product foliation $\Sigma_- \times [0,1]$ along any product annulus or infinite strip $\gamma \times [0,1]$, where $\gamma$ is a connected, embedded 1-manifold in $\Sigma_-$ Poincar\'e dual to $j^+(\Sigma_-)$.
\end{lemma}
\begin{proof}
	First note that $\widehat{U}$ is a foliated interval bundle homeomorphic to $\Sigma_- \times [0, 1]$, i.e. under a homeomorphism, the pullback foliation on $\wh U$ is transverse to the vertical foliation by intervals in $\Sigma_-\times [0,1]$. Thus, the foliation is determined up to conjugacy by the induced holonomy representation 
	\[\rho: \pi_1(\Sigma_1, x) \to \Homeo_+([0, 1]),
	\]
	where $x$ is some basepoint. Moreover, since $U$ is a locally top-depth component, it has trivial holonomy, and hence $\rho$ has the property that $\rho(g)$ is the identity if and only if $\rho(g)$ has an interior fixed point.
    
    From this we see as follows that $\rho(g)$ is abelian and determined precisely by $j^+(\Sigma_-)$. Let $\{x_i \in (0, 1)\}_{i\in \N}$ be a monotonically decreasing sequence so that $(x,x_i)\in\Sigma_-\times[0,1]$ is contained in the same leaf $L$ for all $i$. Then for sufficiently large $i$ and any $g, h \in \pi_1(\Sigma_-, x)$, there is some $n$ so that 
    \[
    \rho([g,h])(x_i) = x_{i + nj^+([g, h])} = x_i
    \]
    for $j^+ = j^+(\Sigma_-)$.

    We see that $\rho$ is conjugate to the representation $j^+(\Sigma)$ under an identification of $(0,1)$ and $\R$ taking the points $L \cap x \times (0,1)$ to the integers. The same argument shows that $\rho$ is determined by $j^-(\Sigma_+)$, except that loops which decrease points go to loops that increase points under the identification. Hence $j^+(\Sigma_-) = -j^-(\Sigma_+)$. 

    Moreover, a foliation on $\Sigma_- \times [0,1]$ determined by shearing along a curve $\gamma$ representing the juncture has a conjugate holonomy map, showing that the foliations are conjugate.
\end{proof}

\section{Forcing closed orbits}\label{sec:forcingclosedorbits}

As we move into the main portion of the paper, we set the following convention.

\begin{convention}
For the rest of the paper, let $M$ be a closed, orientable, atoroidal 3-manifold, and let $\mc F$ be a coorientable 2-dimensional foliation on $M$.
\end{convention}

\subsection{Overview}
In this section we give a sufficient criterion for $\mc F$ to force closed orbits using the dynamics of endperiodic maps. We first review the theory of endperiodic maps and fibered regions in foliations, and then show how this theory quickly leads to orbit forcing. The main result is Theorem \ref{thm:locallytopdepthclosing}, which we state below. We also sketch an alternative perspective using the generalized endperiodic maps of \cite{cantwell_endperiodic_2021}.

\begin{theorem}\label{thm:locallytopdepthclosing}
	Suppose $\mc F$ is finite depth, and let $\phi$ be a transverse flow. If $\mc F$ contains a locally top depth component $U$ that is not a product region, then $\phi$ contains a closed orbit.
\end{theorem}

This result is the main idea behind the positive direction of the Characterization Theorem. We show that unless our techniques for breaking closed orbits apply (Section \ref{sec:plugs}), then $\mc F$ either contains a locally top depth region $U$ as in Theorem \ref{thm:locallytopdepthclosing}, or $\mc F$ is approximated by such foliations. In either case, this produces a closed orbit.

\subsection{Endperiodic and spun pseudo-Anosov maps}

Let $L$ be an infinite-type, orientable, boundaryless surface with finitely many ends, each accumulated by genus, and let $g\colon L\to L$ be a homeomorphism. We say an end $\mc E$ of $L$ is \defn{attracting} under $g$ if there exists a neighborhood $U\subset L$ of $\mc E$ and a power $n\geq 0$ such that 
\[
g^n(U)\subset U,
\] 
and
\[
\bigcap_{k\geq0}g^{kn}(U)=\varnothing.
\]
Similarly, we say $\mc E$ is \defn{repelling} under $g$ if it is attracting under $g^{-1}$. Such neighborhoods $U$ are called \defn{escaping neighborhoods}. We say $g$ is \defn{endperiodic} if there exists some power $k\geq1$ such that every end of $L$ is either attracting or repelling under $g^k$.

Under an endperiodic map $g$, we say a point is \defn{positively escaping} if positive iterates escape out an attracting end for $g$. Similarly, a point is \defn{negatively escaping} if negative iterates escape out a repelling end for $g$.

The connection between endperiodic maps and our work here is the relationship between depth one foliations and compactified endperiodic mapping tori. The \defn{mapping torus} of an endperiodic map $g\colon L\to L$ is as usual defined as
\[
M_g=L\times [0,1]/(p,1)\sim (g(p),0).
\]
The result is an open 3-manifold. By work of Fenley, $M_g$ has a natural compactification \cite{fenley_end_1997}. We present the construction described by Field--Kim--Leininger--Loving in \cite{field2023end}. For each end $\mc E$, let $U_{\mc E}$ be an escaping neighborhood for $\mc E$. Define
\[
\mc U_\pm=\bigcup_{\mc E\in\text{Ends}_\pm}\bigcup_{k\geq0}g^{\pm k}(U_{\mc E}),
\]
where $\text{Ends}_+$ is the set of attracting ends and $\text{Ends}_-$ is the set of repelling ends.

Let $\wh M_g$ be the infinite cyclic cover corresponding to unwinding the circular base of the fiber bundle, so
\[
\wh M_g\cong L\times (-\infty,\infty).
\] 
The action by $\langle g\rangle$ is given by
\[
g\cdot (p,t)=(g(p),t-1)
\]
under these coordinates. Define a partial compactification $\wt M_g$ of $\wh M_g$ by
 \[
 \wt M_g=\mc U_-\times\{-\infty\}\:\cup\:\wh M_g\;\cup\;\mc U_+\times\{\infty\},
 \]
 viewed as living inside the compactification
 \[
 L\times[-\infty,\infty].
 \] The action of $\langle g\rangle$ on $\wh M_g$ extends to $\wt M_g$ by declaring $\pm\infty+1=\pm\infty$ and noting that $\mc U_+$ and $\mc U_-$ are $g$-invariant. Then the compactification $\overline{M_g}$ of $M_g$ is given by the quotient
 \[
 \overline{M_g}=\wt M_g/\langle g\rangle.
 \]
 The boundary is naturally partitioned into 
 \[
 \partial_\pm\overline{M_g}=\mc U_\pm/\langle g\rangle,
 \]
 and is comprised of ideal endpoints of positively (resp. negatively) escaping orbits. 

The compactified mapping torus admits a natural depth one foliation whose depth zero leaves are the boundary components and whose depth one leaves are the fibers of the endperiodic monodromy. In fact, every depth one foliation arises this way (after possibly blowing down interval bundles of compact leaves). 

\begin{lemma}\label{lem:depth one is endperiodic}
    Let $U$ be a depth one region of a depth one foliation, and let $\wh U$ be the path closure of $U$. Then $\wh U$ with its foliation is homeomorphic to a compactified endperiodic mapping torus with its natural foliation. A representative of the monodromy is given by the first return map of any flow transverse to the foliation.
\end{lemma}

\begin{proof}
    Let $U$ be as above with $L$ a leaf in $U$. By Lemma \ref{lem:top depth is fibered}, $U$ fibers over $S^1$ with fiber $L$. Let $\varphi$ be a flow transverse to the depth one foliation. What remains to be shown is just that the monodromy is endperiodic. 

    Each end of $L$ spirals onto a depth zero leaf $\Sigma$ in $\partial\wh U$ in a \defn{spiralling neighborhood}, as in \cite[Section 6]{CC81}. In particular, in some regular neighborhood $N\cong \Sigma\times I$ of $\Sigma$ in $\wh U$, the restriction of $L$ to $N$ looks like an oriented cut and paste of copies of $\Sigma$ of the form $\Sigma\times\{1/n\}$ for $n\in\N$ with an annulus $\gamma\times\I$ for $\gamma$ some essential simple closed curve in $\Sigma$. Take $N$ small enough so that $\varphi$ is trivial, and then from this picture one can see the endperiodic nature of the monodromy.
\end{proof}

We call an endperiodic map \defn{atoroidal} if it does not preserve a finite multicurve up to isotopy. This is equivalent to the mapping torus being atoroidal. As we are always working in an atoroidal manifold, all induced monodromies of taut, depth one foliations will be atoroidal.

Landry--Minsky--Taylor introduced a class of endperiodic maps called \defn{spun pseudo-Anosov maps} in \cite{landry2023endperiodic}. These are homeomorphisms defined as first-return maps of pseudo-Anosov suspension flows almost transverse to a given depth one foliation, embedded within some hyperbolic mapping torus. They show that spun pseudo-Anosov representatives exist in the homotopy class of any atoroidal, endperiodic map \cite[Theorem 6.1]{landry2023endperiodic} and that like pseudo-Anosovs, they have extremal dynamical properties within their homotopy class. The key results for us will be the following.

\begin{theorem}[Landry--Minsky--Taylor \cite{landry2023endperiodic}] Let $f\colon L\to L$ be spun pseudo-Anosov, and let $g$ be an endperiodic map homotopic to $f$. Then for all $n\geq 0$, $g$ has at least as many periodic points of period $n$ as $f$.  
\end{theorem}

A \defn{translation} is an endperiodic map such that every point is both positively and negatively escaping. The following result characterizes translations in terms of their periodic points.

\begin{theorem}[Landry--Minsky--Taylor \cite{landry2023endperiodic}] Let $f$ be a spun pseudo-Anosov map. Then $f$ has no periodic points if and only if $f$ is a translation.
\end{theorem}

We also record the following lemma characterizing translations in terms of their mapping tori.

\begin{lemma}\label{lem:translation mapping torus}
	An endperiodic map $f\colon L\to L$ is homotopic to a translation if and only if 
	\[
	\overline{M_f}\cong \Sigma\times I,
	\]
	for $\Sigma$ a closed surface.
\end{lemma}

\begin{proof}
	First assume $f$ as above is a translation. Then every point in $L$ is both positively and negatively escaping under $f$, so the suspension flow on $\overline{M_f}$ gives it an interval bundle structure. Thus, it must be a product. 
	
	On the other hand, if $\overline{M_f}$ is a product for some map $f$, $f$ must be homotopic to a translation  -- this can be seen by considering the vertical transverse flow and applying Lemma \ref{lem:depth one is endperiodic}.
\end{proof}

\subsubsection{Generalized endperiodic maps}
We will sometimes refer to endperiodic maps on surfaces with infinite endsets, as defined by Cantwell-Conlon-Fenley in \cite{cantwell_endperiodic_2021}. These arise naturally as first-return maps in locally top depth components of depth two or greater. We refer to these as \defn{generalized endperiodics} to avoid confusion. Generalized endperiodic maps have a \defn{soul}, i.e. a surface with boundary and finitely many ends on which the map restricts to an endperiodic map \cite[Proposition 13.30]{cantwell_endperiodic_2021}. 

\subsection{Proof of top-depth component theorem}

We are now ready to prove Theorem \ref{thm:locallytopdepthclosing}, which we restate as Theorem \ref{thm:locallytopdepthclosingagain} below. This will be the basis for the positive direction in the proof of Theorem \ref{thma:classification}, the classification theorem. 

\begin{theorem}\label{thm:locallytopdepthclosingagain}
	Suppose $\mc F$ is finite depth, and let $\phi$ be a transverse flow. If $\mc F$ contains a locally top depth component $U$ that is not a product region, then $\phi$ contains a closed orbit.
\end{theorem}

We first cover the case where $\mc F$ contains a Reeb component as a separate lemma below. Since $M$ is atoroidal, this is equivalent to $\mc F$ not being taut.

\begin{lemma}\label{lem:reebcomponentforces}
    If $\mc F$ has a Reeb component, every flow $\phi$ transverse to $\mc F$ has a closed orbit contained in each Reeb component.
\end{lemma}

\begin{proof}
	Suppose $\mc F$ contains a Reeb component. Let $\phi$ be a flow transverse to $\mc F$, and let $L$ be a planar leaf in the interior of the Reeb component. Observe $L$ is a fiber of a plane bundle over $S^1$. By reversing the direction of $\phi$ if necessary, we assume that the flow points into the Reeb component along its boundary. The flow $\phi$ defines a first return map $f: L \to L$. This is not an endperiodic map as $L$ has only one end; however, $f$ still has controlled end behavior. 
	
	Let $\gamma$ be a simple closed curve in the boundary torus $T$ of the Reeb component which is Poincar\'e dual to the juncture class. Then cutting $L$ along a small flowout annulus of $\gamma$ along $\phi$ separates the complement of a compact disk $D_0$ in $L$ into an infinite union $\{A_i\}$ of annuli glued along their boundaries, which we order by distance from $D_0$. Moreover, for an annulus $A_N$ sufficiently far from $D_0$, $f$ takes $A_N$ to $A_{N-1}$, and so defines a map on the compact disk $D = D_0 \cup A_1 \cup ... \cup A_N$. This map has a fixed point by Brouwer's fixed point theorem \cite{brouwer1911abbildung}. This gives the desired periodic orbit of $\phi$.
\end{proof}

\begin{proof}[Proof of Theorem \ref{thm:locallytopdepthclosing}]
Suppose $\mc F$ is finite depth, and let $\phi$ be a transverse flow. Since $M$ is atoroidal, if $\mc F$ is not taut, then it contains a Reeb component. In this case, Lemma \ref{lem:reebcomponentforces} applies and produces a closed orbit of $\phi$.
	
Suppose instead that $\mc F$ is taut and contains a locally top depth component $U$ that is not a product region. Note that since $\mc F$ is taut, the inclusion of $U$ into $M$ is $\pi_1$-injective. Since $M$ is atoroidal, $U$ must then be as well.	
	
Let $i:N \to M$ be the sutured manifold from Lemma \ref{lem:depth1octopus}. Let $L$ be a leaf of $i^*\mathcal{F}$. Then $L$ is depth 1 in $N$, and by Lemma \ref{lem:depth one is endperiodic}, the pullback flow $\phi_N$ on $N$ defines an endperiodic first return homeomorphism $f_\phi: L \to L$. The flow $\phi$ has a closed orbit in the interior of $N$ if and only if $f_\phi$ has a periodic orbit.

The leaf $L$ will in general have boundary components, which may not be compact. If $L$ has a noncompact boundary component, that boundary component must have one end escape out an attracting end and the other escape out a repelling end. If so, insert an infinite strip of genus along each noncompact boundary component to obtain a surface $L'$. Extend the homeomorphism $f_\phi$ to each strip of genus via a translation map to obtain a map $f'_\phi$. If $L$ has no noncompact boundary components, let $L'=L$ and $f'_\phi=f_\phi$.

Now double $L'$ along its boundary to obtain the infinite-type surface $DL'$. Note that $DL'$ has all ends nonplanar, and the doubled map 
\[
Df'_\phi\colon DL'\to DL'
\]
is endperiodic.

    We claim $Df'_\phi$ is atoroidal. Suppose it is not. Then there is a finite set of essential closed curves $\gamma_i$, $1 \leq i \leq k$ in $DL'$ which are permuted by $Df'_\phi$ up to isotopy. Put each $\gamma_i$ in minimal position with respect to $\partial L$, viewed as a system of essential simple closed curves and properly embedded arcs within $DL'$. Note that $\partial L$ decomposes $DL'$ into two copies of $L$, denoted $L_0$ and $L_1$, and a collection of infinite strips of genus.
    
    First note that no $\gamma_i$ can lie entirely in $L_0$ or $L_1$ inside $DL'$. This is because $Df'_\phi$ preserves $L_0$ and $L_1$, and a periodic curve in $L$ under $f_\phi$ would contradict the fact that $U$ is atoroidal.
    
    On the other hand, no $\gamma_i$ can intersect one of the infinite strips of genus. Suppose some $\gamma_i$ did, and let $\alpha$ be an auxiliary curve entirely inside the infinite strip and with positive geometric intersection number with $\gamma_i$. Since $Df'_\phi$ acts on the strip as a translation, for large $N$ we have 
    \[
    i(\alpha,Df'^N_\phi(\gamma_i))=0.
    \]
    Thus, $Df'^N_\phi(\gamma_i)$ and $\gamma_i$ cannot be homotopic.
    
    Finally, no $\gamma_i$ can intersect a compact boundary component of $L_0$ or $L_1$ inside $DL'$ for the same reason as above, with the boundary component taking the place of $\alpha$. This rules out all possibilities for a periodic curve $\gamma_i$, and we conclude that $Df'_\phi$ is atoroidal.

    Since $Df_\phi'$ is an atoroidal, endperiodic map on a surface with no boundary components and with all ends nonplanar, it is isotopic to some spun pseudo-Anosov map $g$ by \cite{landry2023endperiodic}. By \cite[Lemma 5.13]{landry2023endperiodic}, either $g$ has a periodic point  or $g$ is a translation. The latter case is precisely the case that $N$, and hence $U$, are topological products. By hypothesis we conclude $g$ has a periodic point, and by \cite{landry2023endperiodic} $Df'_\phi$ must also have a periodic point. Since the periodic points of $f'_\phi$ correspond 2-to-1 exactly to those of $f_\phi$, $f_\phi$ has a periodic point and we are done.
\end{proof}

We also record the following lemma, which says that forcing closed orbits is a closed property in the $C^0$ topology. Together with Theorem \ref{thm:locallytopdepthclosingagain}, this proves that if $\mc F$ is approximated by foliations with locally top depth components that aren't products, then $\mc F$ forces closed orbits.

\begin{lemma}\label{lem:approxedimpliesforcing}
	If $\mc F$ is $C^0$-approximated by foliations which force closed orbits, then $\mc F$ forces closed orbits.
\end{lemma}

\begin{proof}
	Let $\mc F$ be a $C^0$ foliation, and let $\{\mc F_i\}_{i\in N}$ be a sequence of foliations which force closed orbits such that $T \mc F_i$ converges to $T \mc F$ in the $C^0$ topology. Let $\phi$ be a flow transverse to $\mc F$. Since transversality is an open condition, past some $N\in \N$ we have that $\phi\pitchfork\mc F_i$. Since $\mc F_i$ forces closed orbits, $\varphi$ has a closed orbit.
\end{proof}

We can now prove Corollary \ref{cor:no holonomy}, restated as Corollary \ref{cor:no holonomy again}.

\begin{corollary}\label{cor:no holonomy again}
	If $\mc F$ is $C^2$ and has no holonomy, then $\mc F$ forces closed orbits.
\end{corollary}

\begin{proof}
	Tischler's Theorem \cite{candel_foliations_1999} implies that $\mc F$ is $C^0$ approximated by fibrations; since $M$ is atoroidal, these fibrations are not fibrations by tori, and so they force closed orbits. Since $\mc F$ is a $C^0$-limit of foliations forcing closed orbits, Lemma \ref{lem:approxedimpliesforcing} implies $\mc F$ forces closed orbits.
\end{proof}

\subsubsection{Alternate proof}
Theorem \ref{thm:locallytopdepthclosing} works in $C^0$ regularity. However, when $\mc F$ is $C^2$, there is a slightly different approach which we sketch here. When $\mc F$ is $C^2$, the monodromy of a locally top-depth component of any depth is a generalized endperiodic.

The soul of this map is essentially the map constructed in Theorem \ref{thm:locallytopdepthclosing}. Since $\mc F$ has compact junctures, choosing the carrying sutured manifold $N$ large enough causes $\mc F$ to restrict to a foliation without holonomy on $\partial_\pitchfork N$, and so the endperiodic map defined in the interior of $N$ is defined on a surface with an infinite union of compact boundary components. Blowing these down to points gives an endperiodic map on which one can apply the results of Landry--Minsky--Taylor \cite{landry2023endperiodic} or the Handel--Miller Theory described in \cite{cantwell_endperiodic_2021}.

\section{Plugs and breaking closed orbits}\label{sec:plugs}

In this section we develop techniques to insert plugs transverse to foliations, thereby breaking closed orbits of transverse flows. The main technical results are Lemma \ref{lem:pluginsertion} and its slight variant Lemma \ref{lem:annularpluginsertion}, which show that plugs can be inserted along curves (or respectively, pairs of parallel curves) of attracting holonomy. These results are enough to prove Theorem \ref{thma:perturb}, the Instability Theorem, which we do in this section. Additionally, we describe the dynamics of the endperiodic first return maps defined by our plug insertions.

\subsection{Plugs Background}

For us, a \defn{plug} will be a sutured manifold $P \cong S \times [-1,1]$, where $S$ is a compact surface with boundary, together with a vector field $X$ on $P$. We set $\partial_+P = S \times \{1\}$, $\partial_-P = S \times \{-1\}$, and $\partial_vP = \partial S \times [-1, 1]$. We require that the pair $(P, X)$ satisfies the following criteria:
\begin{itemize}
    \item[(P1)] The vector field $X$ is positively tangent to the foliation by intervals in a neighborhood of $\partial_v P$ and positively (resp. negatively) transverse to $\partial_+P$ (resp. $\partial_-P$).
    \item[(P2)] There exists a nonempty set $K \subseteq \partial_-P$, the \defn{trapped set}, such that for any $(x, -1) \in K$, the forward orbit of $x$ under $X$ is entirely contained in the interior of $P$.
    \item[(P3)] For any $(y, -1) \in \partial_-P \setminus K$, the forward orbit of $(y, -1)$ exits $\partial_+P$ at $(y, 1)$.
\end{itemize}

A \defn{semiplug} satisfies the same conditions as a plug with the exception of (P3). Concatenating a semiplug with the semiplug obtained by replacing $X$ with $-X$ yields a plug. Our plug constructions will generally be obtained by inserting a semiplug above and below a leaf.

\medskip

The idea behind plugs is that they may be inserted in a manifold with a flow to modify that flow, hopefully by breaking all the closed orbits of that flow. The first condition is to ensure the plug lines up with the original flow at the boundary. The second condition implies that the plug breaks any closed orbits through $K$, and the third condition guarantees that the plug insertion doesn't create any new closed orbits.

\begin{remark}
    Our definition is slightly different from the one appearing in the literature, e.g. \cite{hurder_dynamics_2015}. The primary difference is that we have dropped the \defn{tameness condition}, that the plug embeds into $\R^3$ preserving the vertical direction at the boundary. Our plug insertions will not occur in small coordinate balls, and so this hypothesis is not necessary. Nonetheless, all known plugs satisfy this condition.
\end{remark}

\subsubsection{The Schweitzer plug}
We review the plug which forms the basis for our construction, the \emph{Schweitzer plug} \cite{Schweitzer}. First recall the construction of a Denjoy homeomorphism on the circle $S^1$, which goes as follows. Let $\alpha\in(0,1)$ be an irrational number, and let $\rho_\alpha$ be rotation by $\alpha$ on the circle, where the circle is parameterized as $\R/\Z$. Let $\wh\rho_\alpha$ be the  homeomorphism of $S^1$ obtained from $\rho_\alpha$ by blowing up the orbit through $0$. Then $\wh\rho_\alpha$ is a \defn{Denjoy homeomorphism}. See \cite[Section 12.2]{katok_introduction_1995} for details. 

We parametrize the mapping torus of $\wh \rho_\alpha$, together with the induced suspension flow, as follows. Start with the square 
\[
R=[-1,1]\times[-1,1]
\]
in $\R^2$, and let $X_0$ be the linear unit vector field of slope $\alpha$ on $R$. Identify the top and bottom of $U$, as well as the left and right sides, via translation to obtain a torus $T^2$ with a vector field, which we continue to call $V$. We continue to use $R$ coordinates on $T^2$. Now blow up the orbit through $(0,0)$ of the flow induced by $X_0$ to obtain a new vector field $\wh X_0$. Let $S^1$ be the vertical circle defined by setting the first coordinate of $T^2$ to $-1$. Then $\wh X_0$ is transverse to $S^1$, and the first return map of the induced flow is conjugate to $\wh\rho_\alpha$.

The vector field $\wh X_0$ has a compact minimal set $K_\alpha$ which is the complement of the flowlines that arose from blowing up the orbit through $(0,0)$. Let $U$ be the complement of $K_\alpha$ in $T^2$, so $U$ contains a neighborhood of $(0,0)$.

Let $S=T^2\ssm D$, where $D$ is a small open disk centered at $(0,0)$ and contained in $U$. Let 
\[
\psi: S \times [0, 1] \to [0, 1] 
\] 
be a $C^\infty$ bump function such that 
\[
\psi^{-1}(1) = K_\alpha \times \{1/2\}
\] 
and $\psi = 0$ in a neighborhood of $\partial(S \times [0, 1])$. Let $V$ be the vertical vector field $\frac{\partial}{\partial t}$, where $t$ is the $[0, 1]$ coordinate. Then the \defn{positive Schweitzer semiplug} $P^+_{\alpha, \psi}$ is $S \times [0, 1]$ with the vector field
\[X(y) = (1-\psi(y)) V(y) + \psi(y) \wh X_0(y).\]
We define the \defn{negative Schweitzer semiplug} $P^-_{\alpha, \psi}$ on $S \times [-1, 0]$ with
\[X(y) = (1 - \psi(-y))V(y) - \psi(-y)\wh X_0(y)\]
and the \defn{Schweitzer plug} $P_{\alpha, \psi} = P^-_{\alpha, \psi} \cup P^+_{\alpha, \psi}$. We will suppress the dependence on $\psi$ in the notation when it is not relevant. See Figure \ref{fig:plug} for an illustration of the semiplug.

\begin{figure}[ht]
    \centering
    \includegraphics[width=0.7\linewidth]{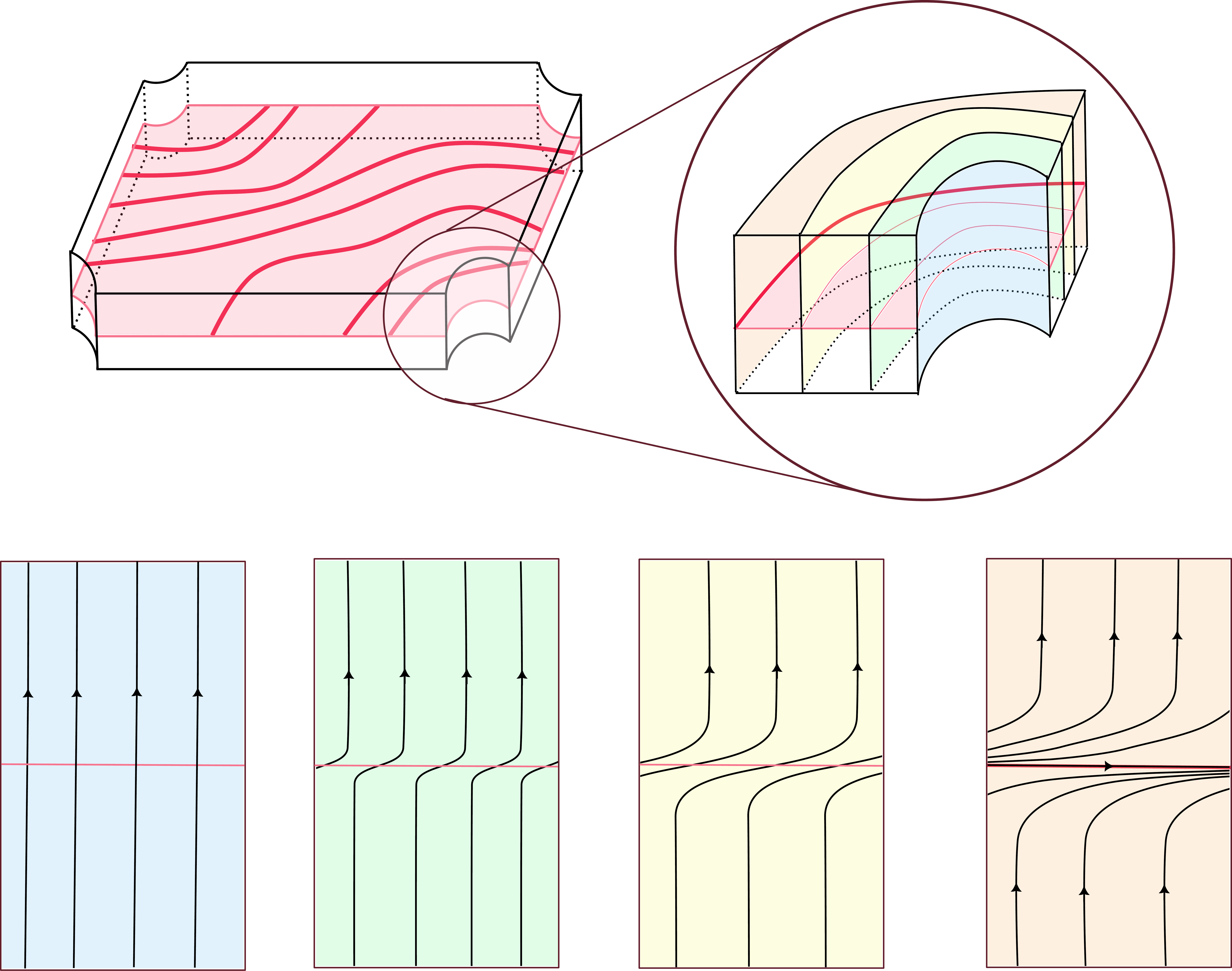}
    \caption{Top row: On the left, a Schweitzer semiplug with the minimal set $K$ in red. On the right, a zoomed in look at the corner with four `windows' highlighted in blue, green, yellow, and orange. The orange window intersects $K$ in an interval. Bottom row: The flow in the four windows, showing how the flow changes as we move toward the trapped set.}
    \label{fig:plug}
\end{figure}

In practice, to insert enough plugs to trap all closed orbits, we require plugs with the property that the trapped set $K$ has nonempty interior, so that we can apply compactness. For this purpose, we define the \defn{modified Schweitzer plug} $P_{\alpha, x, \psi'}$ for $x \in K \subset T^2 \setminus D$, to be the plug obtained by blowing up the orbit through $x$ under the flow induced by $\wh X_0$ and including the blown up orbit in $K$, and then proceeding as in the construction of the Schweitzer plug. This is the suspension of a Denjoy homeomorphism which has been blown up at a second point; a slight modification of the proof that $\hat{\rho}_\alpha$ is $C^1$ (see, e.g. \cite[Proposition 12.2.1]{katok_introduction_1995}) shows that the vector field on $P_{\alpha, x, \psi'}$ is $C^1$. Moreover, this blowup can be chosen very small so that the vector field of $P_{\alpha, x, \psi'}$ is $C^0$-close to that of $P_{\alpha, \psi}$.

\subsection{Plug Insertions}
Let $\phi$ be a flow on $M$ generated by a vector field $V$, and let $P$ be a plug. A \defn{plug insertion} of $(P, X)$ into $\phi$ is an embedding $i: P \to M$ such that $\phi$ is tangent to $X$ in a neighborhood of $\partial P$. By reparametrizing in a small neighborhood of $\partial P$, we obtain the \defn{plugged flow} $\phi_P$ defined by $V$ away from $P$ and $X$ inside $P$ (possibly multiplied by a function supported in a small neighborhood of $\partial P$). We define semiplug insertions similarly. The regularity of the plugged flow $\phi_P$ is the minimum of the regularity of $\phi$ and the regularity of $P$.

The goal of this section is to prove the following lemma, giving sufficient conditions to obtain a modified Schweitzer plug insertion transverse to a foliation $\mathcal{F}$.

\begin{lemma}[Plug Insertion]\label{lem:pluginsertion}
    Let $\phi$ be a flow transverse to $\mathcal{F}$. Let $\gamma_1$ and $\gamma_2$ be two oriented simple closed curves in a leaf $L$ of $\mathcal{F}$ such that:
    \begin{enumerate}
        \item $\gamma_1$ has two-sided attracting holonomy, and
        \item $\gamma_1$ and $\gamma_2$ intersect positively and transversely at a single point.
    \end{enumerate}
    Then for any neighborhood $U$ of $\gamma_1 \cup \gamma_2$ in $M$ for which $\phi$ is smooth, there is a modified Schweitzer plug insertion $i:P = P_{\alpha,x, \psi'} \to U$ such that the plugged flow $\phi_P$ is transverse to $\mc F$. Moreover, if $C$ is any closed disk contained in $U \cap L$, then $i$ can be chosen so that all orbits passing through $C$ are trapped by $P$.
\end{lemma}
By \defn{trapped} we mean that there is a negative function $g: C \to \R^{< 0}$ so that $\phi^{g(x)}(x) \subset i(\partial_-P)$ lies in the image of the trapped set of $P$.

\subsection{Proof of the plug insertion lemma}
We continue the notation of the last section. Let $S = T^2 \setminus D$ have coordinates $[-1, 1] \times [-1, 1]$, with $D$ a small disk about the origin, and let $K_\alpha$ be the minimal set of the suspension of the Denjoy homeomorphism $\hat{\rho}_\alpha$. Set $\gamma_h = [-1, 1] \times \{-1\}$ and $\gamma_v = \{-1\} \times [-1, 1]$. Set their intersection, $(-1, -1)$, to be the basepoint for the fundamental group of $S$, which is freely generated by $\gamma_v$ and $\gamma_h$. Observe that $S$ deformation retracts into any neighborhood of $\gamma_v \cup \gamma_h$.

Recall that a \defn{train track} $\tau$ in a surface is a smoothly embedded finite trivalent graph with a \defn{combing} at each vertex giving it a well-defined tangent space at each point. A \defn{tie neighborhood} of $\tau$ is a normal neighborhood $N_\varepsilon(\tau)$ of $\tau$ foliated by intervals, called \defn{ties}, transverse to $\tau$. A lamination $\lambda$ is \defn{carried} by $\tau$ if for any tie neighborhood $N_\varepsilon(\tau)$, $\lambda$ can be isotoped into $N_\varepsilon(\tau)$ transverse to the ties. It is \defn{fully carried} if it intersects every tie of $N_\varepsilon(\tau)$. For more details, see e.g. \cite[Chapter 1]{calegari_foliations_2007}. Train tracks can also carry other train tracks, using the same definition. 

We construct a family of train tracks $\{\tau_n\}_{n\in\N}$ on $S$ as in Figure \ref{fig:traintrack}. In particular, $\tau_n$ is topologically the union of $\gamma_h$ and a curve $c_n$ freely homotopic to $\gamma_v\gamma_h^n$. Observe that for all $n\geq 1$, $\tau_n$ fully carries $K_\alpha$.

\begin{figure}[ht]
    \centering
    \includegraphics[width=0.85\linewidth]{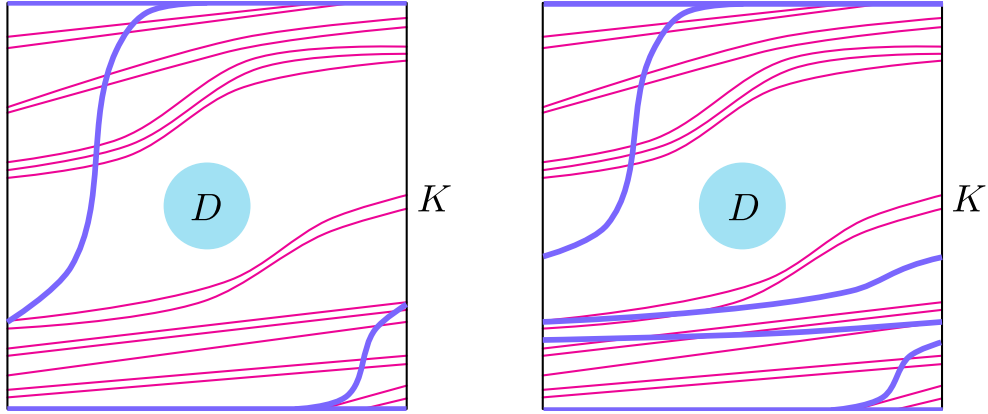}
    \caption{The train tracks $\tau_1$ (left) and $\tau_3$ (right).}
    \label{fig:traintrack}
\end{figure}

We now prove an insertion lemma for the unmodified Schweitzer plug.

\begin{lemma}\label{lem:insertion1}
    Let $\phi$ be a flow transverse to $\mathcal{F}$. Let $\gamma_1$ and $\gamma_2$ be two oriented simple closed curves in a leaf $L$ of $\mathcal{F}$ such that
    \begin{enumerate}
        \item $\gamma_1$ has attracting holonomy on the positive side, and
        \item $\gamma_1$ and $\gamma_2$ intersect positively and transversely at a single point $x$.
    \end{enumerate}
    Then for any neighborhood $U$ of $\gamma_1 \cup \gamma_2$ in $M$ for which $\phi$ is smooth, there is a positive Schweitzer semiplug insertion $i:P = P^+_{\alpha, \psi} \to M$ into $\phi$ such that 
    \begin{enumerate}
        \item the plugged flow $\phi_P$ is transverse to $\mathcal{F}$,
        \item the image $i(P)$ is contained in $U$, 
        \item the image of negative boundary $i(\partial_-P^+_{\alpha, \psi})$ is contained in $L$, and
        \item the orbit of $\phi$ through $x$ is trapped by $P$.
    \end{enumerate}
    If $\gamma_1$ in addition has attracting holonomy on the negative side, there is a Schweitzer plug insertion $i': P = P_{\alpha, \psi} \to M$ satisfying the first two conditions above.
\end{lemma}

\begin{proof} Let $\phi$, $\gamma_1$, $\gamma_2$, and $L$ be as in the lemma statement. Fix a neighborhood $U$ of $\gamma_1 \cup \gamma_2$ in $M$. Pick $\varepsilon$ small enough such that \[\phi^{[0,\varepsilon]}(\gamma_1 \cup \gamma_2) \subset U,\] and
 the holonomy along $\gamma_1$ is defined and attracting along the entire transversal $\phi^{[0, \varepsilon]}(x)$ for any $x$ in $\gamma_1$. 
 We now take a small tubular neighborhood $V \subset U$ of $\gamma_1 \cup \gamma_2$ in $L$ so that $\phi^{[0, \varepsilon]}(\overline{V}) \subset U$. By switching the orientation of $\gamma_2$ if necessary, we can choose a diffeomorphism 
    \[
    h\colon S\to \overline{V}
    \]
    taking $\gamma_h$ to $\gamma_1$ and $\gamma_v$ to $\gamma_2$. This gives an embedding 
    \[
    i_0\colon P_0\to U,
    \] 
    where $P_0=S\times[0,1]$, by the formula 
    \[ 
    i_0(x, t) = \phi^{\varepsilon t}(h(x)). 
    \]
    We pull back $\mathcal{F}$ along $i_0$ to obtain a foliation $\mc G$ on $P_0$. We henceforth work in $P_0$ and consider the holonomy along curves or arcs in $S$ as acting on relatively open intervals $[0, 1]$ containing $0$.
    
    Observe that $\mc G$ has attracting holonomy along $\gamma_h$. Therefore, there is some $n$ such that $\gamma_v\gamma^n_h$ has attracting holonomy in the interval $[1/4, 3/4]$. Fixing $\alpha \in (0, 1)$, and recall that the train track $\tau_n$ carries $K_\alpha$. 

	We now define a lift of $\tau_n$ from $S$ to $P_0$, transverse to $\mc G$, as follows. First recall that $\tau_n = \gamma_h \cup c_n$, and $\gamma_h\cap c_n$ is an arc $a$ not containing $(-1,-1)$. We begin by lifting $(-1, -1)$ to $(-1, -1, 1/2)$ and proceed in the forward direction along $\gamma_h$ using parallel transport for $\mc G$, obtaining a lift of $\gamma_h$ to a leaf of $\mc G$ which is continuous except at $(-1,-1)$. This is possible since $\mc G$ has attracting holonomy along $\gamma_h$ on the entire interval $[0,1]$. Let 
	\[
	f\colon \gamma_h\to[0,1]
	\]
	be the height function for this lift.
	
	Next, isotope $\tau_n$ in $S$ such that $a^-$, the negative endpoint of the arc $a$, is close enough to $(-1,-1)$ to satisfy
	\[
	f(a^-)\in(1/4,3/4).
	\]
	Now start lifting $c_n$ to $P_0$, starting by lifting $a^-$ to $(a^-,f(a^-))$, and continuing along $c_n$ in the forward direction via parallel transport for $\mc G$ as with $\gamma_h$. This extends the lift of $\gamma_h$ to a lift of $\tau_n$, contained in a single leaf of $\mc G$, where we now have discontinuities over both $(-1,-1)$ and $a^-$. Abusing notation, we now let $f$ denote the height function of this lift on all of $\tau_n$.
	
	To remedy these discontinuities, and to make the lift of $\tau_n$ transverse to $\mc G$, we `tilt' the lift up as we move in the forward direction along $\gamma_h$ and $c_n$ as follows. Since $\gamma_h$ has attracting holonomy at $1/2$, the limit of $f$ approaching $(-1, -1)$ in the forward direction along $\gamma_h$ is less than $1/2$. Similarly, since $f(a^-)\in(1/4,3/4)$, the holonomy along $c_n$ is defined and attracting at $f(a^-)$, so the limit of $f$ approaching $a^-$ in the forward direction along $c_n$ is less than $f(a^-)$. Thus, we can define a function
	\[
	E\colon \tau_n\to[0,1]
	\]
	which is $0$ at $(-1,-1)$ and $a^-$, has discontinuities at exactly these points, is otherwise smooth and monotonely increasing along the positive directions of $\gamma_h$ and $c_n$, and such that
	\[
	f+E\colon \tau_n\to [0,1]
	\]
	gives a height function for a smooth lift of $\tau_n$ to $P_0$ transverse to $\mc G$. Let
	\[
	\tilde{f}\colon\tau_n\to P_0
	\]
	denote this lift.

    Extend $\tilde{f}$ to a small tie neighborhood $N_\delta = N_\delta(\tau_n)$ of $\tau_n$ in $S$ so that the graph of $N_\delta$ is transverse to $\mathcal{G}$ and the intersection foliation defines a foliation by ties of $N_\delta$. Let $$q: S \to S$$ be a diffeomorphism of $S$ isotopic to the identity which takes $K_\alpha$ into $N_\delta$ transverse to this foliation by ties. Then the graph of $q(K_\alpha)$ under $\tilde{f}$ is transverse to $\mc G$. Let $$r: P^+_{\alpha, \psi} \to P_0$$ be a diffeomorphism which preserves the interval fibers in the vertical direction and takes $K_\alpha \times \{1/2\}$ to the graph of $q(K_\alpha)$ under $\tilde{f}$, and define $i = i_0 \circ r$. Then for a bump function $\psi$ with sufficiently small support, $$i: P^+_{\alpha, \psi} \to U$$ defines the desired semiplug insertion. Note that the point $x$ lies in $\tau_n$, and so we may choose the isotopy of $K_\alpha$ into a neighborhood of $\tau_n$ so that $x$ lies in $K_\alpha$.

    In the case that $\gamma_1$ has attracting holonomy on both sides, choose $n$ large enough that $\gamma_2\gamma^n_1$ has an interval of attracting holonomy on both sides, and perform the same construction to obtain a negative semiplug insertion $i_-: P^-_{\alpha, \psi} \to U$. Choose the same $\psi$ for both constructions. By modifying $i_-$ via an isotopy near $L$, the images of $K_\alpha$ can be chosen to match up and the union of the positive and negative semiplugs defines a full Schweitzer plug insertion.
\end{proof}

From this we prove Lemma \ref{lem:pluginsertion}, the insertion lemma for the modified Schweitzer plug that traps an entire disk's worth of orbits. We restate this as Lemma \ref{lem:pluginsertion again} below.

\begin{lemma}[Plug Insertion]\label{lem:pluginsertion again}
	Let $\phi$ be a flow transverse to $\mathcal{F}$. Let $\gamma_1$ and $\gamma_2$ be two oriented simple closed curves in a leaf $L$ of $\mathcal{F}$ such that:
	\begin{enumerate}
		\item $\gamma_1$ has two-sided attracting holonomy, and
		\item $\gamma_1$ and $\gamma_2$ intersect positively and transversely at a single point.
	\end{enumerate}
	Then for any neighborhood $U$ of $\gamma_1 \cup \gamma_2$ in $M$ for which $\phi$ is smooth, there is a modified Schweitzer plug insertion $i:P = P_{\alpha,x, \psi'} \to U$ such that the plugged flow $\phi_P$ is transverse to $\mc F$. Moreover, if $C$ is any closed disk contained in $U \cap L$, then $i$ can be chosen so that all orbits passing through $C$ are trapped by $P$.
\end{lemma}

\begin{proof}
	Let $\phi$, $\gamma_1$, $\gamma_2$, $L$, $U$, and $C$ be as in the lemma statement. Let $R$ be a compact region in $U\cap L$ containing $\gamma_1$, $\gamma_2$, and a small open neighborhood around $C$, and homeomorphic to $T^2\setminus D$. Similar to the proof of Lemma \ref{lem:insertion1}, let $\varepsilon>0$ be small enough so that 
	\[
	\phi^{[0,\varepsilon]}(R)\subset U,
	\]
	and the holonomy along $\gamma_1$ is defined and attracting along the entire transversal $\phi^{[0,\varepsilon]}(z)$ for any $z$ in $\gamma_1$.  Proceeding with the rest of the proof of Lemma \ref{lem:insertion1}, taking $V\subset U$ to contain $R$, produces an unmodified Schweitzer plug insertion
	\[
	i_0\colon P_{\alpha, \psi} \to M
	\]
	contained in $U$, and whose image contains $C$.
	
	In $P_{\alpha,\psi}$, we may blow up $K$ around $x$ to obtain a modified Schweitzer plug insertion 
	\[ 
	i\colon P = P_{\alpha, x, \psi} \to M.
	\]
	Choose this blowup to be small enough that the plugged flow is still transverse to $\mathcal{F}$. This is possible since transversality is an open condition. 
	
	What remains to be shown is that we can change $i$ so that all orbits through $C$ are trapped by the new plug. To achieve this, we `funnel' the orbits through $C$ into the trapped set of $P$ as follows. Let $C'$ be a small closed disk contained in the trapped set of $P$. 
	Let $D$ be an open disk in $L\cap i(P)$ containing $C$ and $C'$, and fix a neighborhood $B\cong D\times[-1,1]$ such that the vertical $[-1,1]$ fibers are flowlines of $\phi$. By reparametrizing the flowline coordinates of $B$ and extending $P$ along flowlines, we may assume that the disks $D \times \{\pm 1/2\}$ form the positive and negative boundary of $i(P) \cap B$ and that $\mc F$ restricts to a product foliation in $D \times [1/2 - \delta, 1/2]$ and $D \times [-1/2, -1/2 + \delta]$, where $\delta>0$ is chosen small enough that these regions are outside of $i(\text{supp}(\psi))$.
	
	Let $\pi\colon B\to D$ be the projection map along flow lines, and let 
	\(
	h_t\colon D\to D
	\)
	for $t\in[0,\delta]$ be an isotopy such that
	\begin{enumerate}
		\item $h_0$ is the identity,
		\item $h_\delta(\pi(C))\subset \pi(C')$,
		\item $h_t=h_0$ for $t$ close to $0$, and 
		\item $h_t=h_\delta$ for $t$ close to $\delta$.
	\end{enumerate}
	
    Then define a map $H:B\to B$ by 
    \[H(z, t) = \begin{cases}
    (h_{ t - (\frac 1 2 - \delta)}(z),t) & t \in [1/2 - \delta, 1/2]\\
    (h_{t - (\delta - \frac 1 2)}(z),t) & t \in [-1/2, -1/2 + \delta]\\
    (z, t) & \text{ otherwise.}
    \end{cases}\]
    Then $H \circ i: P \to M$ is the desired plug insertion.
\end{proof}

\subsection{Annuli of attracting holonomy}\label{sec:annulusholonomy}
It will sometimes be convenient to work with a slight generalization of Lemma \ref{lem:pluginsertion} in which we use an annulus of attracting holonomy in place of a curve. Suppose $\mc F$ has two leaves $L_-, L_+$ which cobound a product region 
$V$ on the positive side of $L_-$ and the negative side of $L_+$. Let $A$ be an annulus of attracting holonomy with boundary components $\partial_- A$, $\partial_+ A$ lying in $L_-$ and $L_+$, respectively. 

If the boundary components of $A$ are nonseparating, then there is an embedded curve $\gamma$ in $L_-$ intersecting $\partial_- A$ transversely at a single point. Any normal neighborhood of $\partial_- A \cup \gamma$ yields a product sutured manifold homeomorphic to $S^1 \times( T^2 \setminus D)$ with negative and positive boundary components on $L_-$ and $L_+$, respectively. If $\phi$ is a flow transverse to $\mc F$ which defines a dynamical product between $L_-$ and $L_+$ and which is tangent to $A$, this sutured manifold can be defined by flowing along $\phi$. 

Inserting a positive semiplug along the positive boundary of this sutured manifold and a negative semiplug along the negative boundary of this sutured manifold gives a full plug insertion into $\phi$. Another way of seeing this is to blow down the product region and identify $L_+$ and $L_-$, insert a plug, and then pull back the new flow by inserting a dynamical product. This discussion, along with the proof of the Plug Insertion Lemma, proves the following.

\begin{lemma}[Annular Plug Insertion Lemma]\label{lem:annularpluginsertion}
	Suppose $L_-$ and $L_+$ cobound a topological product $L_- \times [0, 1]$ and $\varphi$ is a flow transverse to $\mc F$ which defines a dynamical product between $L_-$ and $L_+$. Suppose $A$ is an annulus of attracting holonomy between them which is tangent to $\phi$. Let $\gamma$ be an oriented simple closed curve in $L_-$ such that $\partial_- A$ and $\gamma$ intersect positively and transversely at a single point.
	Then for any neighborhood $U$ of $\gamma \times [0, 1] \cup A$ in $M$ for which $\phi$ is smooth, there is a modified Schweitzer plug insertion 
	\[
	i:P = P_{\alpha,x, \psi'} \to U
	\] 
	such that the plugged flow $\phi_P$ is transverse to $\mc F$. Moreover, if $C$ is any closed disk contained in 
	\[
	N\left(\gamma \times [0, 1] \cup A\right) \cap L_-,
	\]
	where $N\left(\gamma \times [0, 1]\cup A\right)$ is a normal neighborhood of $\gamma \times [0, 1]\cup A$, then the insertion can be chosen so that all orbits passing through $C$ are trapped by $P$.
\end{lemma}

\subsection{Proof of the Instability Theorem}

We now prove the Instability Theorem (Theorem \ref{thma:perturb}), restated as Theorem \ref{thm:perturb again} below.

\begin{theorem}\label{thm:perturb again}
	 If every leaf of $\mc F$ has genus, then $\mc F$ can be $C^0$-approximated by foliations which do not force closed orbits.
\end{theorem}

\begin{proof}
    Suppose every leaf of $\mc F$ has genus, and let $\phi$ be a smooth flow transverse to $\mathcal{F}$. By compactness of $M$, we may find a finite collection $\{C_j\}_{j=1,\ldots,n}$ of closed disks such that
    \begin{enumerate}
        \item each $C_j$ is contained in a distinct leaf $L_j$ of $\mc F$, and
        \item every closed orbit of $\phi$ intersects at least one $C_j$.
    \end{enumerate}
    Let $\{D_j\}_{j=1,\ldots,n}$ be a finite collection of open disks such that 
    \[
    C_j\subset D_j\subset L_j
    \]
    for all $j$, and flow out each $D_j$ slightly along $\phi$ to a disjoint collection $\{V_j\}$ of foliated flowbox neighborhoods. For each $j$, let $\gamma_j$ be a nonseparating simple closed curve in $L_j$. Such a curve is guaranteed to exist since each leaf of $\mc F$ has genus.

    We now construct a new foliation $\mc F'$ by blowing up each $L_j$ and inserting the following foliation. On $L_j \times [-1, 1]$, consider the product foliation by copies of $L_j$, and take a simple closed curve $\alpha_j$ in $L_j \times \{0\}$ which intersects $\gamma_j \times \{0\}$ transversely in a single point. Such an $\alpha_j$ exists since $\gamma_j$ is nonseparating. Now shear along the annuli $\alpha_j \times [-1, 0]$ and $\alpha_j \times [0, 1]$ so that $\gamma_j \times \{0\}$ has attracting holonomy. Call this foliation $\mc F_j$. Then $\mc F'$ is obtained as a generalized Denjoy blowup of $\mc F$ along the leaves $L_j$, inserting the foliations $\mc F_j$ as saturated sets $U_j$ for all $j$. By \cite[Theorem 5.2]{kazez_c0_2017}, this foliation can be chosen to be $C^0$-close to $\mc F$. In particular, under this construction $\phi$ will still be transverse to $\mc F'$. 
    
    Abusing notation, we refer to the leaf in $\mc F'$ corresponding to $L_j \times \{0\}$ in $F_j$ as $L_j$. Furthermore, we consider $D_j$ and $\gamma_j$ as lying in $L_j \times \{0\}$. Note the collection $\{C_j\}$ still intersects each closed orbit of $\phi$. 
    
    By Lemma \ref{lem:pluginsertion} we may insert modified Schweitzer plugs $\{P_j\}_{j=1,\ldots,n}$ which trap all the closed orbits passing through $C_j$. These insertions can be chosen to lie in the blowup regions $U_j$ and therefore are disjoint. The resulting plugged flow, $\phi'$, will be the desired $C^1$ aperiodic flow.
\end{proof}

\subsection{First return maps}\label{subsec:firstreturn}
We now discuss the dynamics of first return maps induced by the semiflow on a Schweitzer semiplug. This material is independent from  the remainder of the paper. We focus on the case of an unmodified Schweitzer; the first return map for a modified Schweitzer plug is obtained from this by blowing up an orbit in the invariant Cantor set. 

Let $U$ be a locally top-depth region which is a product region, which has a transverse dynamical product flow $\phi$, and suppose we have a semiplug insertion $i: P^+ \to \widehat{U}$ with $i(\partial_-P^+) \subset \partial_-\widehat{U}$. Let $\phi_P$ be the plugged flow; this flow defines a first return map $f: L \to L$ on any leaf $L$ of $U$. Since $\phi_P$ is a dynamical product outside of the image of $(T^2 \setminus D) \times [0, 1]$, we need only describe the first return map on the foliation restricted to the image of the semiplug. This foliation has relative depth one with the relative depth zero leaf being a punctured torus, so each leaf in the interior is homeomorphic to 
\[
L_P = (S^1 \times \R) \setminus( \{1\} \times \Z),
\] 
a cylinder with a $\Z$'s worth of points removed. We continue to denote the plugged first return map on $L_P$ by $f$.

Let $C$ be the intersection of $L_P$ with the minimal set $K$ of the plug. This set is homeomorphic to a Cantor set, and $f$ acts on it as a restriction of the Denjoy homeomorphism used to define the plug. The best way to understand the dynamics of $f$ on $L_P$ is to consider a separating simple closed curve $\alpha$ in $L_P$ that is `parallel' to $C$, as in Figure \ref{fig:first return map}, and look at its iterates under $f$. 

\begin{figure}[ht]
    \centering
\includegraphics[width=0.75\linewidth]{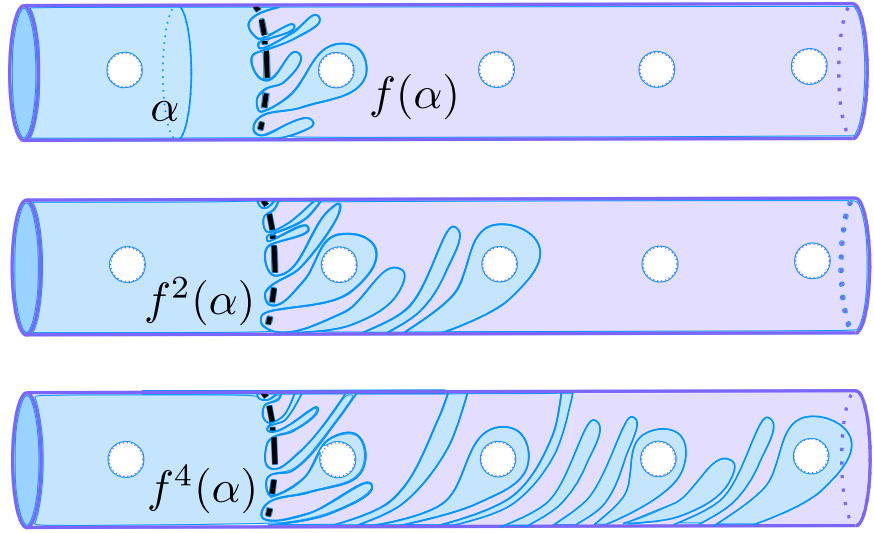}
 \caption{The first return map $f\colon L_P\to L_P$ defined by the Schwietzer semiflow. The black lines represent the Cantor set $C$, and the hole punches are where the rest of $L$ meets $L_P$.}
    \label{fig:first return map}
\end{figure}

What we see is that $\alpha$ gets hung up on the Cantor set $C$ but pushes through the gaps, eating one more puncture of $L_P$ after each iteration. There is also some twisting, coming from the Denjoy homeomorphism. The result is a map with an invariant Cantor set but no periodic points. Figure \ref{fig:first return map} is our best illustration of what this looks like.

Since $C$ is part of $L_P$ (and not part of its space of ends), the map $f$ is homotopic to a translation, and the interesting dynamics of this map are not `homotopically interesting.' In particular, Handel--Miller theory is not applicable, and the mapping torus of $f$ is just the mapping torus of a translation on $L_P$. But, removing the Cantor set from $L_P$ fixes this problem. So, let $L_P'=L_P\setminus C$, and let $f'$ be the restricted homeomorphism. This map is now reminiscent of a suspension of a Denjoy homeomorphism on the Cantor tree surface, as studied in \cite{hernandez2026fibrations}, except there is now endperiodic behavior involved as well. Motivated by \cite[Theorem 1.1]{hernandez2026fibrations} and the fact that endperiodic mapping tori are tame, we end this section with the following questions.

\begin{question}
    What is the mapping torus of $f'\colon L_P'\to L_P'$? In particular, is it tame (i.e. homeomorphic to the interior of a compact 3-manifold)?
\end{question}

\subsection{Other plugs}\label{sec:other plugs}
We end this section with a brief digression on the use of other plugs. We first observe that Harrison's construction of ``Schweitzer-suspension'' vector fields \cite{HARRISON1988249} can always be chosen to be transverse to a foliation: in each semiplug, such a vector field $V$ is transverse to a properly embedded annulus $A$ in $T^2 \setminus (D \times [0, 1])$, and so ``spinning'' $A$ around $T^2$ as in Lemma \ref{lem:connectingapproxfinite} or \cite{landry2023endperiodic} yields a depth 1 foliation in $T^2 \setminus (D \times [0, 1])$ which is transverse to $V$. Doubling this foliation yields a foliation transverse to the entire plug, which has a curve of two-sided attracting holonomy.

In particular, this shows that the original Wilson plug, Harrison's $C^2$ plug, and G. Kuperberg's $C^1$ volume-preserving plug are all transverse to a relative depth 1 foliation. More is true for the latter two plugs: since both are $C^0$ perturbations of the Schweitzer plug, the proof of Lemma \ref{lem:pluginsertion} shows that they can be inserted along nonseparating curves with two-sided attracting holonomy.

The difficulty with extending Theorem \ref{thma:classification} by constructing $C^2$ or $C^1$ volume-preserving counterexamples is that the plug used in our argument must trap an open set. It is not clear whether Harrison's construction can be done with a doubly blown up Denjoy diffeomorphism as in the modified Schweitzer plug above. Additionally, no volume-preserving plug can trap an open set, and G. Kuperberg's volume-preserving ``twisted Wilson plugs'' do not work in this setting. Therefore, constructing aperiodic volume-preserving flows transverse to foliations would require genuinely new ideas.

However, since Wilson's plug traps an open set and can be inserted into a model foliation, one can increase the regularity of Theorem \ref{thma:perturb} to $C^2$ by first perturbing the foliation to add Wilson's smooth plug to trap an open set, and then shearing inside the blowup regions to insert Harrison's plug.

On the other hand, extending past $C^{3-\varepsilon}$ regularity seems substantially more difficult. It is not clear how to construct a foliation transverse to K. Kuperberg's $C^\infty$ plug, and it seems plausible none exists. We mention also the following conjecture of Harrison, which would imply there are no $C^3$ plugs obtained by Schweitzer suspension:
\begin{conjecture}[Loxodromic Mapping Conjecture \cite{HARRISON1988249, Harrison_Pugh_1992}]
    Suppose $f: S^2 \to S^2$ is a $C^3$ diffeomorphism which fixes the north and south poles, attracting the first and repelling the second, and has no other fixed points. If one orbit is asymptotic to both the north and south poles, then $f$ has no other minimal invariant sets.
\end{conjecture}

\section{Connecting junctures and approximating finite depth foliations}\label{sec:junctures}

\begin{convention}
For this section, we additionally assume that $\mc F$ is finite depth and $C^2$-smoothable.
\end{convention}

\subsection{Overview}

In Section \ref{sec:forcingclosedorbits}, we showed that if $\mc F$ has a nonproduct locally top depth region, or is approximated by such foliations, then $\mc F$ forces closed orbits. In Section \ref{sec:plugs}, we showed that if $\mc F$ has a curve or annulus of attracting holonomy, then we can insert a plug transverse to $\mc F$, thereby breaking some closed orbits of a transverse flow. Our goal in this section is to show that exactly one of these situations holds. Together with the right choice of initial transverse flow to plug, which we handle in Section \ref{sec:constructingflows}, this will prove the Characterization Theorem.

We prove this dichotomy by studying regions we call \emph{base regions}, which generalize locally top depth components by allowing for iterated shearing. Within a base region, we introduce the idea of having \emph{connecting (limit) junctures}. Whether or not a base region has connecting (limit) junctures determines whether it is approximated by fibrations or instead has a curve or annulus of attracting holonomy. The capstone result of this section is the following.

\begin{proposition}\label{prop:junctures capstone}
	Let $U$ be a base region of $\mc F$. Then the following dichotomy holds.
	\begin{enumerate}
		\item If $U$ has connecting (limit) junctures, then $\mc F$ can be approximated by fibrations in $U$.
		\item If $U$ does not have connecting (limit) junctures, then it contains a nonseparating curve of two-sided attracting holonomy or an annulus of attracting holonomy. 
	\end{enumerate}
\end{proposition}

This is used to prove the Characterization Theorem as follows. If any base region $U$ of $\mc F$ falls into case (1), we apply Theorem \ref{thm:locallytopdepthclosing} to show that $\mc F$ is approximated by foliations that force closed orbits in $U$, which in turn shows $\mc F$ forces closed orbits in $U$. If instead, every base region satisfies case (2), we apply the Plug Insertion Lemma to insert modified Schweitzer plugs along the curves produced in case (2). By starting with the right choice of transverse flow, the plugged flow will be aperiodic.

\subsection{Base regions and connecting junctures}\label{subsec:base regions}
Starting with our foliation $\mc F$, we produce a new foliation $\check{\mc F}$ as follows. For every locally top-depth region of $\mc F$ that is a product region, replace the foliation in this region to make it a trivially foliated product. 
This may produce new such regions, so continue the process until all product regions are trivially foliated products. This new foliation is $\check{\mc F}$. 

Define a \defn{base region} of $\mathcal{F}$ to be a locally top-depth region of $\check{\mc F}$. If $U$ is a base region, it is also an open saturated set for $\mc F$. Define the \defn{base depth} of a base region $U$ to be the depth of $U$ in $\check{\mc F}$. If $U$ has depth $k$ and base depth $\ell$, note that the relative depth of $U$ is $k - \ell$. 

Define a \defn{mapping torus region} of $\mc F$ to be an open, saturated set (topologically) fibering over $S^1$ with fiber any relative depth zero leaf. Define the \defn{fiber depth} of a mapping torus region to be the depth of a relative depth zero leaf.
 
We first record the following preliminary result describing the structure of base regions.

\begin{lemma}\label{lem:base region structure}
	Let $U$ be a base region of base depth $\ell$ for the foliation $\mc F$. For $d\geq \ell$, let 
	\[
	V_d=\{L\;\mid\;L\text{  a leaf in }U\text{ of depth }\geq d\}.
	\]
	Then the following hold.
	\begin{enumerate}
		\item $U$ is not a product region.
		\item $\mc F\mid_U$ is obtained from $\check{\mc F}\mid_U$ by iterated shearing in product regions.
		\item If $d>\ell$, $V_d$ is a collection of open product regions $V_d^i$, each of the form \[\Sigma^i\times (0,1)\] for $\Sigma^i$ a leaf of $\mc F\mid_U$ of depth $d-1$, and the foliation on $V_d^i$ is obtained from the product foliation by iterated shearing. Furthermore, $V_d^i$ is a mapping torus region.
	\end{enumerate}
	
\end{lemma}

\begin{proof}
	Throughout, let $U$ be a base region of base depth $\ell$ of the foliation $\mc F$, and let $V_d$ be defined as in the lemma statement. 
	\begin{enumerate}
		\item Suppose $U$ were a product region of the form $L\times (0,1)$. By construction of $\check{\mc F}$, we must have that the foliation $\check{\mc F}\mid_U$ is by parallel copies of $L$. However, since $U$ is a topological product, we also have that the boundary leaves of $\wh U$ in $\check{\mc F}$ are homeomorphic to $L$. Thus, $U$ is not a locally top depth component of $\check{\mc F}$, contradicting the definition of a base region.
		
		\item The foliation $\check{\mc F}$ is obtained inductively by replacing locally top-depth product regions with trivially foliated products. We work inductively in reverse: suppose $V$ is a locally top-depth product region, of $\mc F$, so $\wh V \cong L \times [0, 1]$, with $\mc F$ restricting to a relative depth one foliation with no holonomy in the interior. Then the next step of the induction is obtained by replacing $\mc F$ in $V$ with the product foliation by copies of $L$. By Lemma \ref{lem:juncturesinproduct}, the reverse direction is given by shearing. 
		
		\item Let $\phi$ be a flow transverse to $\mc F$ which defines a dynamical product on each topological product. Let $L_d$ be a depth $d$ leaf of $U$. Then since $U$ fibers over $S^1$, it is not hard to see that $\phi$ defines a first return map on $L_d$, thereby giving a homeomorphism $\wh{U \setminus L_d} \cong L_d \times [0, 1]$ in which $\mc F$ is transverse to the interval fibers. This shows that any connected component of a saturated set $V\subsetneq U$ is a product region. In particular, each connected component of $V_d$ is a product region. To see that it is open, observe that a limit of points lying in leaves of depth $k$ must be in a leaf of depth $\leq k$, and so the complement of $V_d$ is closed. Now the same argument as prior shows that each $V^i_d$ is obtained by iterated shearing. To see that $V^i_d$ is a mapping torus region, observe that in $V_d$, $\mc F$ is obtained by iterated shearing, and hence by replacing the topological product $V_{d+1}$ with a trivially foliated product, we obtained a foliation in which $V_d$ is a locally top-depth component. Therefore $V_d$ topologically fibers over $S^1$ with fibers the depth $d$ leaves.
	\end{enumerate}
\end{proof}

By the above lemma, a base region is a mapping torus region, and in this case the fiber depth agrees with the base depth.

We now work in a single base region $U$. Suppose $U$ has depth $k$ (i.e. the maximal depth of all leaves of $\mc F$ in $U$ is $k$) and suppose it has base depth $\ell$. For each $i \in \ell, ..., k-1$, we associate to a depth $i$ leaf $L$ the positive and negative junctures $j^\pm(L)$ from the spiralling of depth $i + 1$ leaves onto $L$. If $L$ lies in a trivially foliated product $L \times [0, 1]$, we identify the cohomology of $L \times \{0\}$ and $L\times \{1\}$ so that we may consider both positive and negative junctures as lying in $L$.

\begin{definition} Using the notation above, say $L$ has \defn{connecting junctures} if 
\[
j^+(L) = -j^-(L).
\]
Say a mapping torus component $V$ in $U$ has \defn{connecting junctures} if every leaf realising the fiber depth does. Say $U$ has \defn{connecting junctures} if all mapping torus components of $U$ do.
\end{definition}

See Figure \ref{fig:connectingjunctures} for an illustration of a leaf with connecting junctures.

\begin{figure}[ht]
	\centering
	\includegraphics[width=0.5\linewidth]{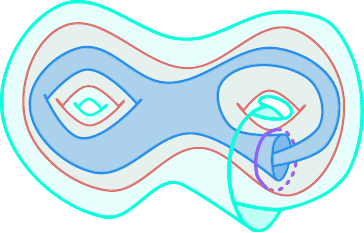}
	\caption{The central orange surface has connecting junctures.}
	\label{fig:connectingjunctures}
\end{figure}

\medskip

The next two sections tackle Proposition \ref{prop:junctures capstone} in two cases: the case where $\mc F$ has only finitely many locally top-depth regions, and the case where it has infinitely many. In the latter case we introduce a generalization of the notion of a juncture, which we call a \emph{limit juncture}.

\subsection{Connecting junctures, finite case}
Let $\mathcal{F}$ have depth $n$, and suppose throughout this section that $\mathcal{F}$ has finitely many locally top-depth regions. Our goal in this subsection is to prove Proposition \ref{prop:junctures capstone} with this assumption. We prove it in full generality in the next subsection, but it is easier to first demonstrate the ideas in the case of finitely many locally top-depth regions. 

\medskip

We first show that having connecting, compactly supported junctures bounds the relative depth of a base region. The $C^2$ hypothesis in Theorem \ref{thma:classification} is used to guarantee $\mc F$ has compactly supported junctures so that we may use the following result.

\begin{proposition}\label{prop:C2connectinghasdepth1}
	Suppose $U$ is a base region of depth $k$ and base depth $\ell$. If $U$ has connecting, compactly supported junctures, then $U$ has relative depth at most one.
\end{proposition}
\begin{proof}
	Let $U$ be a base region of depth $k$ and base depth $\ell$, and suppose $U$ has compactly supported junctures. Suppose that $U$ has relative depth higher than one, i.e. $k-\ell>1$. 
	
	Let 
	\[
	V_{k-1}=\{L\;\mid\;L\text{ a leaf in }U\text{ of depth }\geq k-1\},
	\]
	as in Lemma \ref{lem:base region structure}. Suppose for simplicity that $V_{k-1}$ is connected; the same argument applies to each connected component of $V_{k-1}$. Since $k-1>\ell$, we have that $V_{k-1}$ is not all of $U$, and is instead a product region of the form 
	\[
	\Sigma\times(0,1)
	\]
	for $\Sigma$ a leaf of $\mc F\mid_U$ of depth $k-2$ by Lemma \ref{lem:base region structure}. Furthermore, $V_{k-1}$ fibers over $S^1$ with fiber $L_{k-1}$, a leaf of depth $k-1$ in $V_{k-1}$.
	
	Since $U$ has compactly supported junctures, the monodromy 
	\[
	f\colon L_{k-1}\to L_{k-1}
	\]
	for the fibration structure above is a generalized endperiodic map. Furthermore, $f$ is homotopic to a translation, since the mapping torus is a topological product. Thus, $f$ has no eigenvectors on $H^1_c(L_{k-1})$, the compactly supported cohomology of $L_{k-1}$.
	
	On the other hand, we now claim that if $U$ has connecting junctures, then the map $f$ \emph{does} have an eigenvector on $H^1_c(L_{k-1})$. Suppose toward contradiction that $U$ has connecting junctures. Within $V_{k-1}$, the depth $k$ leaves are contained in saturated mapping tori of the form 
	\[
	L_{k-1}\times (0,1),
	\]
	and these are divided by depth $k-1$ trivially foliated products of the form
	\[
	L_{k-1}\times [0,1].
	\]
	Label these regions as 
	\[
	W_1',W_1,W_2',W_2,\ldots,W_r', W_r,
	\]
	where the $W_i'$ are the depth $k-1$ regions, the $W_i$ are the depth $k$ regions, and consecutive indices correspond to adjacent regions, ordered cyclically to agree with coorientation. 
	
	Let $\phi$ be a transverse flow on $V_{k-1}$ turning each $W_i'$ into a dynamical product, and use this flow to blow down each $W_i'$ into a single leaf
	\[
	L^i\cong L_{k-1}.
	\]
	Note now that the positive and negative boundary components of $\wh W_i$ are $L^i$ and $L^{i+1}$ respectively, with indices taken mod $r$.
    
	For each $i$, let $\gamma_i$ be a representative in $H_1(L^i)$ of the Poincar\'e dual of $j^+(L^i)$. This $\gamma_i$ can be chosen to be a compact, immersed 1-manifold since the junctures are compactly supported. Let 
	\[
	A_i = \gamma_i \times [0,1],
	\]
	where we identify $\wh W_i$ with $L^i \times [0, 1]$. Then $A_i$ is an immersed annulus with boundary 
	\[
	-\gamma \times \{0\} \sqcup \gamma \times \{1\}.
	\]
	Its positive boundary is a representative of $j^-(L^{i+1})$. Denote $\partial_- A_i = -\gamma \times \{0\}$ and $\partial_+ A_i = \gamma \times\{1\}$.
	
	Since the junctures connect, modulo $n$ we have that 
	\[
	[\partial_- A_i] = [\partial_+ A_{i-1}] \in H_1(L^i).
	\] 
	Therefore we can connect them with compact immersed subsurfaces $S^i$ contained in $L^i$. These close up to a closed, homologically nontrivial surface in $V$, so the Poincar\'e dual of the juncture class provides an eigenvector of the monodromy on $H_1(L_{k-1})$. Since the junctures are compactly supported, this gives an eigenvector in $H^1_c(L_{k-1})$. This produces the desired contradiction.
\end{proof}

\begin{remark}
Keeping the notation of the proof above, a similar argument shows that if $V_{k-1}$ has connecting junctures which are not compactly supported, then the monodromy has an eigenvector on $H^1(L_{k-1})$. However, translations have many such eigenvectors, and so this phenomenon does indeed occur. The surface representatives now lie in $H_2(\wh V_{k-1}, \partial \wh V_{k-1})$.
\end{remark}

\begin{figure}[ht]
	\centering
	\includegraphics[width=0.55\linewidth]{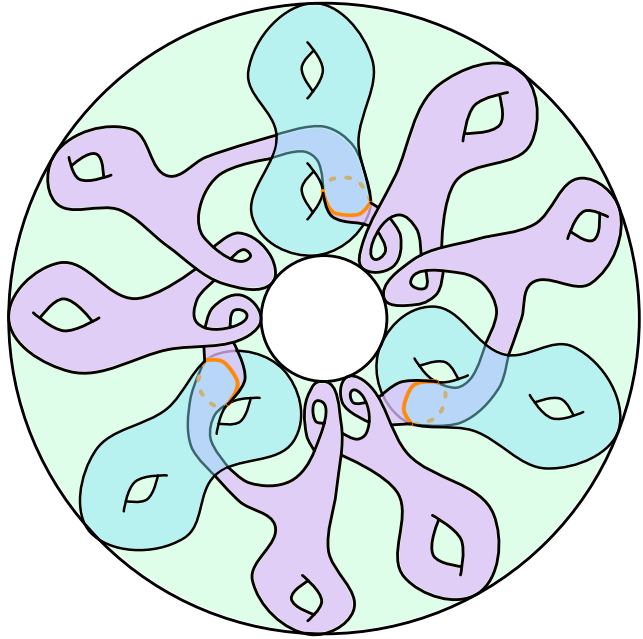}
	\caption{A fiber leaf arising from connecting junctures.}
	\label{fig:despinning}
\end{figure}

Building on Proposition \ref{prop:C2connectinghasdepth1}, we now show that if a base region $U$ has relative depth at most one, then the foliation can be approximated by fibrations within that base region. The idea behind this is the `despinning construction' as in \cite{Gabai1987, landry2023endperiodic}. Figure \ref{fig:despinning} illustrates the simplest example, where the base region $U$ is all of $M$, the depth zero leaves are fiber surfaces for a fibration structure on $M$, and the depth one leaves form product regions between the depth zero leaves. Choose a leaf in each depth one region, take a sufficiently large subsurface, and push this subsurface so that the boundary is contained in the depth zero leaves cobounding the depth one region. The connecting junctures hypothesis implies that these subsurfaces can be made to connect up, forming a single fiber surface shown in purple in the figure. By taking larger and larger subsurfaces, one obtains foliations by fibers that approximate the original depth one foliation.

\begin{lemma}\label{lem:connectingapproxfinite}
	Suppose $U$ is a base region of $\mc F$ of relative depth at most one. Then in $U$, $\mc F$ can be approximated by fibrations.
\end{lemma}

\begin{proof}
	If $U$ is a base region of relative depth zero, it is a locally top-depth component and the result follows from Lemma \ref{lem:top depth is fibered}. So, let $U$ be a base region of $\mc F$ of relative depth one. Fix a smooth flow $\phi$ transverse to $\mc F$ and defining a dynamical product on every product region. Let $N$ be an immersed carrying sutured manifold for $U$, as given by the Octopus Decomposition Lemma \ref{lem:octopus}. Enlarging $N$ if necessary, we assume the restriction of $\mc F$ to $\partial_\pitchfork N$ is a trivial foliation by circles. This is possible since there are finitely many locally top-depth components in $U$, and each one has compact junctures. We work in $\wh N$, where $\mc F$ pulls back to a foliation $\mc G$. Since $U$ is relative depth one, the foliation $\mc G$ is depth two.

	As above, let $L_0, ..., L_n$ be depth one fiber leaves separating the top-depth regions $W_1, ..., W_n$ in $\mc G$. For each positive juncture $j^+(L_i)$, choose a compact, embedded 1-manifold realization $m_i$. Since the junctures connect, we can take $-m_i$ as a representative of $j^-(L_i)$. 
	
	We now truncate the depth 2 leaves in each $W_i$ and push them into $L_i$ as follows. Let $S_i$ be a fiber leaf of $W_i$, take a large subsurface $S_i'\subset S_i$, and push the negative boundary slightly along $\phi$ into $L_i$ and push the positive boundary into $L_{i+1}$ (modulo $n$). Each $S_i$ has finitely many ends parallel to each end of $L_i$, some number of boundary components meeting $\partial_\pitchfork N$ tangent to $\mc G$, and compact negative and positive boundary components  lying in $L_{i}$ and $L_{i+1}$, respectively. As in \cite[Lemma 3.7]{landry2023endperiodic}, following ideas of Gabai \cite[Lemma 0.6]{Gabai1987} , we may choose these $S_i'$ so that the boundary is precisely $-m_i \cup m_{i+1}$, and so each $S_i'$ glues up (along with product annuli $m_i \times [0,1]$ for each $L_i$ which lies in a product region) to obtain a compact surface $\Sigma$. 
	
	Taking a sequence of such $\{S_{i_j}'\}_{j\in\N}$ exhausting each $S_i$, we obtain a sequence $\{\Sigma_j\}_{j\in\N}$ of surfaces as above with boundary tangent to $\mc G$ in $\partial_\pitchfork N$. For sufficiently large $j$, each $\Sigma_j$ intersects every orbit of $\phi$ infinitely many times and hence is a fiber of a fibration $\mc G_k$ of the interior of $N$, tangent to $\mc G$ along $\partial_\pitchfork N$. The boundary condition allows us to modify $\mc F$ in the interior of $N$ to a foliation $\mc F_j$ which is a fibration in $U$. The foliations $\mc F_j$ $C^0$-converge to $\mc F$, as desired.
\end{proof}

We now handle what happens when $\mc F$ has a base region without connecting junctures. We show below that this results in a curve of two-sided attracting holonomy or an annulus of attracting holonomy. When a base region of $\mc F$ does not have connecting junctures, the curve or annulus of attracting holonomy produced by Lemma \ref{lem:connectingholonomyfinite} will allow us to insert a plug transverse to $\mc F$ in a neighborhood of that curve or annulus, as discussed in Section \ref{sec:plugs}.

\begin{lemma}\label{lem:connectingholonomyfinite}
	Suppose $\mc F$ is $C^2$ and has no toral leaves, and let $U$ be a base region of $\mc F$. If $U$ does not have connecting junctures, then $U$ has a leaf $L$ with a nonseparating curve of two-sided attracting holonomy or has an annulus of attracting holonomy.
\end{lemma}

\begin{proof}
	Let $\mc F$ and $U$ be as in the lemma statement, and let $L$ be a depth $i$ leaf realizing the fact that $U$ does not have connecting junctures, i.e. $L$ satisfies
	\[
	j^+(L) \neq -j^-(L)\;\;\text{in}\; \; H^1(L,\R),
	\]
	possibly after identifying $L$ with a leaf cobounding a trivially foliated product with $L$. If this is the case, blow down that product so that $L$ is isolated in the set of depth $i$ leaves.
	
	Since $\mc F$ is $C^2$ and has no toral leaves, all leaves of positive depth have all ends accumulated by genus by \cite[Corollary 8.4.7]{candel_foliations_1999}. In particular, no matter the depth, the dimension of $H_1(L, \R)$ is at least 2. The classes $j^+(L)$ and $j^-(L)$ both define linear functionals on $H_1(L, \R)$. The open half-spaces 
	\[
	P^\pm = \{a \in H_1(L, \R) \mid j^\pm(a) > 0\}
	\]
	have nonempty intersection if and only if $j^+(L) \neq -j^-(L)$. Therefore there is a nonempty open cone $C = P^+ \cap P^-$ such that any simple closed curve with homology class lying in $C$ has attracting holonomy.
	
	We now claim we can find such a curve which is nonseparating. Since the junctures of $\mc F$ are compactly supported, any simple closed curve whose homology class lies in $C$ must have positive signed count of intersections with a compact representative of the Poincar\'e duals of the junctures $j^\pm(L)$, and so is nonseparating. Since $C$ is open, we can find a primitive integral homology class to obtain the desired curve.
	
	If $L$ did not lie on the boundary of a trivially foliated product, we have produced a curve of two-sided attracting holonomy and are done. Otherwise, blow the product region back up and observe that the preimage under the blowup of the curve of attracting holonomy is an annulus of attracting holonomy.
\end{proof}

We end this subsection with the following observation. If $U$ is a base region with junctures that don't connect, then Theorem \ref{thma:classification} implies that it is not approximated by fibrations (by Lemma \ref{lem:approxedimpliesforcing}, combined with the construction of aperiodic flows). However, there is an easier way to see this, following ideas of \cite{andrade_foliations_2003}.

\begin{lemma}\label{lem:holonomy implies not fibr}
	Suppose a base region $U$ has a curve or annulus of two-sided attracting holonomy. Then it is not approximated by fibrations.
\end{lemma}

\begin{proof}
	Let $U$ be a base region for a foliation $\mc F$, and let $\gamma$ be a curve of attracting holonomy (or be a core curve of an annulus of attracting holonomy). Suppose toward contradiction that $\mc F\mid_U$ is approximated by a sequence $\{\mc F_i\}_{i\in\N}$ of fibrations in $U$. 
	
	By the assumption on $\gamma$, it is isotopic in $U$ to a curve positively transverse to $\mc F$ and hence positively transverse to $\mc F_i$ for $i$ sufficiently large. Likewise, it is isotopic in $U$ to a curve negatively transverse to $\mc F$ and hence negatively transverse to $\mc F_i$ for $i$ sufficiently large. This is a contradiction since the $\mc F_i$ are fibrations in $U$.
\end{proof}

\subsection{Connecting junctures, infinite case}
We now modify the techniques and definitions of the previous subsection to handle the general case, where $\mc F$ has infinitely many locally top-depth regions. Observe that if $\mc F$ has infinitely many locally top-depth components, then the compactness of $M$ implies all but finitely many of them are topological products. In particular, $\mc F$ has only finitely many base regions. For convenience, we set the following convention.

\begin{convention}
	For the remainder of this section, fix a smooth flow $\phi$ transverse to $\mc F$ defining a dynamical product on every product region.
\end{convention}

Let $U$ is a base region with infinitely many locally top-depth components. We develop a juncture theory for $U$ and show that the same dichotomy holds as in the finite case, namely that in $U$ either $\mc F$ is approximable by fibrations, or contains a curve or annulus of attracting holonomy. As was the case for the finite case, we work in mapping torus components. Let $V$ be a mapping torus component of $U$ with fiber depth $\ell$.  

Let $\mc A = \{\Sigma_a\}_{a \in A}$ be the set of depth $\ell$ leaves in $V$. Since these leaves are fibers of a fibration over $S^1$, we topologize the index set $A$ as a subset of $S^1$. The index set $A$ is compact. For convenience we assume $A$ is totally disconnected. If $A$ contains a compact interval, we will do as in the finite depth case, identifying endpoint leaves and working only with the `outside pointing' junctures on the endpoint leaves. The set $A$ inherits a local order from the coorientation of $\mc F$, that is, every $a \in A$ has a neighborhood which is totally ordered, and this order is inherited by subsets and compatible on intersections.

Let $\Sigma_a$ for $a \in A$ be a depth $\ell$ leaf of $\mc F$. Suppose $a \in A$ is isolated on the negative side (given by the local order around $a$). Then we define the \defn{negative juncture} $j_-(\Sigma_a)$ as usual, as the juncture of the depth one leaves spiraling onto $\Sigma_a$ on the negative side. As we are assuming $\mc F$ is $C^2$-smoothable, junctures are compact, so  $j_-(\Sigma_a) \in H^1_c(\Sigma_a)$.

Now suppose $a$ is not isolated on the negative side. For sufficiently small $\varepsilon$, say smaller than some $\varepsilon_0$, our transverse flow $\phi^t$ defines a half-neighborhood 
\[
N_\varepsilon = \phi^{(-\varepsilon, 0]}(\Sigma) \cong \Sigma \times (-\varepsilon, 0]
\] 
where $\mc F$ restricts to a foliation transverse to the interval fibers. Then sliding along $\phi$ induces a homeomorphism $p_b: \Sigma_a \to \Sigma_b$ for any depth $\ell$ leaf $\Sigma_b\subset N_\varepsilon$.

Define
\[J^-_\varepsilon(\Sigma_a) = \{p_b^*(j_-(\Sigma_b)) \mid \Sigma_b \subset N_\varepsilon \text{ is isolated on the negative side}\},\]
and define the \defn{negative limit juncture}
\[J^-(\Sigma_a) = \bigcap_{ 0>\varepsilon>\varepsilon_0}J^-_\varepsilon (\Sigma_a) \subset H^1(\Sigma_a).\]
Define the \defn{positive limit juncture} $J^+(\Sigma_a)$ the same way, except take $\varepsilon > 0$. Note that $J^\pm(\Sigma_a)$ does not depend on the choice of $\varepsilon_0$ appearing in the definition.

Since $A$ is compact and totally disconnected, it is homeomorphic to a subset of the Cantor set, and therefore the set of points which are negatively (resp. positively) isolated is dense.

Note that if $\Sigma_a$ is isolated on the negative (resp. positive) side, $J^-(\Sigma_a) = \{j^-(\Sigma_a)\}$ (resp. $J^+(\Sigma_a) = \{j^+(\Sigma_a)\}$), and so the definition of limit juncture recovers the standard definition of juncture..

\begin{definition}
	Say that the mapping torus component $V$ has \defn{connecting (limit) junctures} if for all $a \in A$, 
	\[
	J^+(\Sigma_a) = -J^-(\Sigma_a)\hspace{1cm}\text{and}\hspace{1cm}|J^+(\Sigma_a)| = |J^-(\Sigma_a)| = 1.
	\]
	Say the base region $U$ has \defn{connecting (limit) junctures} if all mapping torus components of $U$ do.
\end{definition}

We begin with a preliminary lemma.
\begin{lemma}\label{lem:connectingimpliesconstant}
	Suppose $U$ has relative depth one and has connecting limit junctures, and let $N$ be a closed topological product with distinct relative depth zero negative and positive boundary leaves $\Sigma_-$ and $\Sigma_+$. Then under the homeomorphism $\Sigma_- \cong \Sigma_+$ induced by sliding along flow lines of $\phi$ in $N$, $J^-(\Sigma_-) = -J^+(\Sigma_+)$.
\end{lemma}

\begin{proof}
	We can index the relative depth zero leaves of $\mc F$ in $N$ by a set $A_N$, which we consider as a subset of a flowline of $\phi$ with the subspace topology, which we in turn identify with the interval $[0, 1]$.
	
	We consider the limit junctures $J^\pm$ as functions $A_N \to H^1(\Sigma_-)$. We claim that by the connecting limit junctures hypothesis, for every $x \in A_N$ there is an open neighborhood $B_x$ such that $J^-$ (and hence also $J^+$) is constant. Let $\{y_i\}_{i\in\N} \subseteq A_N$ be a sequence converging to $x$. By passing to a subsequence, either all $y_i < x$ or all $y_i > x$. First suppose $y_i < x$. Then there is a sequence $z_i \to x$ such that each $z_i$ is isolated on the negative side and $J^-(z_i) = J^-(y_i)$. Now since $J^-(x)$ is a singleton, by the definition of limit junctures $J^-(z_i)$ is eventually constant.
	
	On the other hand, suppose $x < y_i$ for all $i$. Then there is a sequence $w_i$ converging to $x$ such that each $w_i$ is isolated on the negative side and $J^-(y_i) = J^-(w_i)$ for all $i$. Each $w_i$ cobounds a gap with some point $t_i$, and so by Lemma \ref{lem:juncturesinproduct}, $J^-(w_i) = j^-(w_i) = -j^+(t_i) = -J^+(t_i)$. However, for $t_i$ close enough to $x$, since all junctures are singletons, $-J^+(t_i) = -J^+(x) = J^-(x)$, applying the connecting junctures hypothesis. Therefore, for sufficiently large $i$, $J^-(y_i) = J^-(x)$. This shows that $J^-$ is locally constant.
	
	To conclude that a function on a subset of a Cantor set in $[0, 1]$ is constant, it suffices to show that it is both locally constant and that its values on endpoints of gaps are equal. But the latter claim follows again by Lemma \ref{lem:juncturesinproduct} and the connecting junctures hypothesis. Therefore $J^-$ (and hence also $J^+$) is constant on $A_N$. This concludes the proof.
\end{proof}

We now proceed as in the finite case, first showing that the $C^2$ hypothesis allows us to assume we are working in relative depth at most one. We then show this implies the foliation is approximable by fibrations, and finally that when this is not the case, we can find a curve or annulus of attracting holonomy.

\begin{proposition}\label{prop:connecting implies rel depth 1}
	Suppose $U$ has compactly supported junctures and connecting limit junctures. Then $U$ has relative depth at most one.
\end{proposition}
\begin{proof}
	The proof is exactly the same as in the finite case; the connecting limit junctures hypothesis allows us to build a homology class in a connected relative depth one region $V$ topologically fibering over $S^1$ with generalized endperiodic monodromy. Choose a single locally top-depth region and use Lemma \ref{lem:connectingimpliesconstant} to insert annuli in the other product regions with boundaries representing the junctures of the finitely many chosen locally top-depth regions. This gives the desired homologically nontrivial closed surface in $V$.
\end{proof}

\begin{lemma}\label{lem:connectingapproxinfinite}
	Suppose $U$ is a base region of relative depth at most one, has compactly supported junctures, and has connecting limit junctures. Then in $U$, $\mc F$ can be approximated by fibrations.
\end{lemma}
\begin{proof}
	Let $U$ be as in the lemma statement. We show that in $U$, $\mc F$ can be $C^0$-approximated by a foliation with finitely many locally top-depth components and connecting junctures. In some Riemannian metric, since $U$ has infinitely many locally top-depth regions, for any $\varepsilon > 0$ it has locally top-depth regions $V$ such that every orbit segment of $\phi$ contained in $V$ has length less than $\varepsilon$. Then in all such $V$ we replace $\mc F$ with a product foliation. As $\varepsilon \to 0$, we get a sequence $\mc F_\varepsilon$ of foliations with finitely many locally top-depth regions in $U$ which approaches $\mc F$.
	
	For a fixed $\varepsilon$, let $V_\varepsilon$ be the union of all of these new product regions and let $N$ be a connected component of its closure. For $\varepsilon$ sufficiently small, $N$ is not all of $U$. Then $N$ is a trivially foliated product region, and its positive and negative boundary components $\Sigma_+$ and $\Sigma_-$ bound gaps of $A$, the index set for the lowest depth leaves in $U$. By Lemma \ref{lem:connectingimpliesconstant}, $J_-(\Sigma_-) = -J_+(\Sigma_+)$ in $\mc F$, and since we have only replace leaves in the interior of $N$ this also holds for $\mc F_\varepsilon$. Therefore, we see that $\mc F_\varepsilon$ has connecting junctures. Since $\mc F_\varepsilon$ has finitely many locally top-depth regions, we may apply Lemma \ref{lem:connectingapproxfinite} to conclude that in $U$, $\mc F_\varepsilon$ is approximated by fibrations. Finally, a diagonal argument shows $\mc F$ is approximated by fibrations in $U$.
\end{proof}

\begin{lemma}\label{lem:connectingholonomyinfinite}
	Suppose $U$ does not have connecting limit junctures. Then $U$ has a nonseparating curve or annulus of attracting holonomy.
\end{lemma}
\begin{proof}
	If $U$ does not have connecting limit junctures, there is some leaf $S$ with $J^+(S) \neq -J^-(S)$. If $S$ is isolated, this is the standard juncture and the argument is the same as in Lemma \ref{lem:connectingholonomyfinite}. If $S$ is not isolated, then by the definition of limit junctures there are leaves $S_-, S_+$, with $S_- \leq S \leq S_+$ in the local coorientation order, such that $S_-$ (resp. $S_+$) is isolated on the negative (resp. positive) side, and under the identifications of their homologies in the product neighborhood they cobound, $j^-(S_-) \neq -j^+(S_+)$. The same argument as in the finite case gives a nonseparating annulus of attracting holonomy.
\end{proof}

\subsection{Wrapping up}
Finally, we put the pieces together to prove Proposition \ref{prop:junctures capstone}, restated as Proposition \ref{prop:junctures capstone again} below.

\begin{proposition}\label{prop:junctures capstone again}
	Suppose $\mc F$ is finite depth and $C^2$, and let $U$ be a base region of $\mc F$. Then the following dichotomy holds.
	\begin{enumerate}
		\item If $U$ has connecting (limit) junctures, then $\mc F$ can be approximated by fibrations in $U$.
		\item If $U$ does not have connecting (limit) junctures, then it contains a nonseparating curve of two-sided attracting holonomy or an annulus of attracting holonomy. 
	\end{enumerate}
\end{proposition}

\begin{proof}
	Suppose $\mc F$ is finite depth and $C^2$, and let $U$ be a base region. Suppose first that $U$ has connecting (limit) junctures. Then Proposition \ref{prop:connecting implies rel depth 1} implies $U$ has relative depth one, and Lemma \ref{lem:connectingapproxinfinite} applies to show $\mc F$ is approximated by fibrations in $U$.
	
	On the other hand, if $U$ does not have connecting (limit) junctures, then Lemma \ref{lem:connectingholonomyinfinite} produces the desired curve or annulus of attracting holonomy.
\end{proof}

\section{Constructing aperiodic flows}\label{sec:constructingflows}

In this section, we prove the Characterization Theorem (Theorem \ref{thma:classification}), as well as Corollary \ref{cor:blowdown}. The main difficulty is that we must be very efficient in our use of plugs, and so we must choose a starting flow with `sparse' closed orbits, which can be trapped easily by a small number of plugs. In this section we review \emph{Smale flows}, which are the starting flows we will use. 

\subsection{Smale flows}
A flow $\phi$ is a \defn{Smale flow}, after \cite{franks_nonsingular_1985}, if the chain-recurrent set of $\phi$ is hyperbolic, has dimension at most one, and all stable and unstable manifolds intersect transversely.

The relevant criterion for us is that dimension of the chain-recurrent set is 1 (it will never be 0 for a nonsingular flow). This implies that the intersection of the chain-recurrent set with any transverse surface is closed and totally disconnected. The chain-recurrent set, in particular, contains all periodic orbits.

For the proof of Theorem \ref{thma:classification}, we will need the following result.

\begin{theorem}\label{lem:smaleisdense}
    Any smooth vector field is $C^0$-close to a smooth vector field generating a Smale flow.
\end{theorem}

The case of diffeomorphisms is due to Shub \cite{shub_structurally_1972}, who presented a simple modification of Smale's weaker result \cite{smale_stability_1973} that any diffeomorphism is isotopic to a Smale diffeomorphism. The analogue of Smale's result for flows is due to de Oliveira \cite{de_oliveira_non-singular_1980}, and it can be seen from the proof that a similar modification to Shub's result can be performed in this setting, i.e. by performing the isotopy relative to a sufficiently small triangulation. This is also sketched in an expository preprint by Zeeman \cite{zeeman_c0_1972}.

\subsection{Proof of the Characterization Theorem}
We wrap up by proving the remaining results of the paper, namely the Characterization Theorem (Theorem \ref{thma:classification}) and its corollary, Corollary \ref{cor:blowdown}. For convenience, we restate these as Theorem \ref{thm:characterization again} and Corollary \ref{cor:blowdown again} below. Note that we have changed the wording of Theorem \ref{thma:classification}, but this does not change the content, as any foliation is approximated by itself.

\begin{theorem}[Characterization Theorem]\label{thm:characterization again}
	Suppose $\mc F$ is finite depth and $C^2$. Then $\mc F$ forces closed orbits if and only if $\mc F$ is $C^0$-approximated by foliations containing locally top-depth components that are not product regions. 
	If $\mc F$ does not force closed orbits, it admits a $C^1$ transverse aperiodic flow.
\end{theorem}

\begin{proof}
    Suppose $\mc F$ is finite depth and $C^2$. Let $U_i$ for $i=1,\ldots,n$ be the base regions of $\mc F$. If any $U_i$ has connecting (limit) junctures, then by Proposition \ref{prop:junctures capstone}, $\mc F$ can be approximated by fibrations $\{\mc F_j\}_{j\in\N}$ in $U_i$. By Lemma \ref{lem:base region structure}, $U_i$ is not a product region. Thus, Theorem \ref{thm:locallytopdepthclosing} implies each $\mc F_j$ forces closed orbits. Finally, Lemma \ref{lem:approxedimpliesforcing} implies $\mc F$ forces closed orbits in $U_i$, as it is approximated there by foliations that force closed orbits. 
    
    Suppose instead that each $U_i$ has junctures which don't connect. Then there is a leaf $L_i$ in each $U_i$ which either has a nonseparating curve $\gamma_i$ of attracting holonomy, or cobounds a topological product with a leaf $L_i'$ and has an annulus of attracting holonomy which restricts to a nonseparating curve in $L_i$. Let $\phi_0$ be  a smooth flow which defines a dynamical product on every topological product region in $\mc F$. For each $i$, any closed orbit of $\phi_0$ intersecting $U_i$ must intersect every leaf of $U_i$. This is because $\phi_0$ defines a suspension flow on $U_i$ with fiber any relative depth zero leaf, and all higher depth leaves accumulate on these leaves. Note also that every closed orbit of $\phi_0$ must intersect some $U_i$, as open sets in the complement of the base regions are all contained in product regions. From this we see that every closed orbit of $\phi_0$ intersects some $L_i$. Then by Lemma \ref{lem:smaleisdense}, we can $C^0$-approximate $\phi_0$ by a Smale flow $\phi$ which must also have the property that every closed orbit intersects some $L_i$.
    
    Let $C$ be the closure of the set of closed orbits of $\phi$. Then the set $C$ intersects each $L_i$ in a union of totally disconnected sets. Since $M$ is compact, each $L_i$ has a compact subsurface $S_i$ such that every closed orbit of $\phi$ intersects some $S_i$. We can choose $S_i$ so that its boundary is disjoint from $C$, and therefore we can find a closed disk $D_i$ in each $S_i$ that contains $C \cap S_i$ in its interior. In particular, every closed orbit of $\phi$ intersects some $D_i$.

    If $L_i$ has a curve of attracting holonomy, then by Lemma \ref{lem:pluginsertion} we may insert a plug $P_i$ in a small neighborhood of $L_i$ so that $D_i$ is trapped by the plug. Choose the neighborhood small enough so that the image of $P_i$ is contained in the base region $U_i$.

    If $L_i$ cobounds a topological product region with $L_i'$ and has an annulus of attracting holonomy. Then since $\phi$ defines a dynamical product on every product region, we may similarly use Lemma \ref{lem:annularpluginsertion} to insert a plug $P_i$ trapped all orbits intersecting $D_i$.

    Let $\phi'$ be the plugged flow after inserting the plugs $P_i$ for all $i$. These plug insertions are disjoint, so no new closed orbits are created with each plug insertion. Every orbit in $C$ is trapped by some $P_i$. Therefore, $\phi'$ is an aperiodic flow transverse to $\mc F$. Since we have inserted $C^{1}$ plugs into a $C^\infty$ flow, $\phi'$ is generated by a $C^{1}$ vector field, as desired. This concludes the proof.
\end{proof}

\begin{corollary}\label{cor:blowdown again}
	Suppose $\mc F$ is finite depth. If $\mc F$ does not force closed orbits, then it is obtained by iterated shearing of a foliation that does.
\end{corollary}

\begin{proof}
	Suppose $\mc F$ is finite depth. Let $\check{\mc F}$ be the foliation defined in Section \ref{subsec:base regions}, i.e. the foliation obtained from $\mc F$ by iteratively replacing locally top-depth components that are topological products with trivially foliated products. Then by Lemma \ref{lem:base region structure}, every locally top-depth component of $\check{\mc F}$ is not a topological product, and so $\check{\mc F}$ forces closed orbits by Theorem \ref{thm:locallytopdepthclosing}. By Lemma \ref{lem:base region structure}, $\mc F$ is obtained from $\check{\mc F}$ by iterated shearing, and this completes the proof.
\end{proof}

\bibliographystyle{alpha}
\bibliography{bibliography_cited_only}

\end{document}